\documentclass[10pt]{amsart}
\usepackage{amsmath,amscd,amssymb,amsfonts}
\usepackage{mathrsfs}
\usepackage{enumerate}
\usepackage{hyperref}
\numberwithin{equation}{section}       
\numberwithin{figure}{section}       

\theoremstyle{plain}
\newtheorem{theo}{Theorem}
\newtheorem{prop}{Proposition}[section]

\newtheorem{coro}[prop]{Corollary}
\newtheorem{lemm}[prop]{Lemma}

\newtheorem{theoalph}{Theorem}

\theoremstyle{definition}
\newtheorem{defi}[prop]{Definition}

\theoremstyle{remark}
\newtheorem{rema}[prop]{Remark}

\newtheoremstyle{citing}
  {3pt}
  {3pt}
  {\itshape}
  {}
  {\bfseries}
  {.}
  {.5em}
  {\thmnote{#3}}

\theoremstyle{citing}

\newcommand{\partn}[1]{{\smallskip \noindent \textbf{#1.}}}

\newcommand{\C}{\mathbb{C}}

\newcommand{\R}{\mathbb{R}}

\newcommand{\Z}{\mathbb{Z}}

\newcommand{\cL}{\mathcal{L}}
\newcommand{\cM}{\mathcal{M}}

\newcommand{\fD}{\mathfrak{D}}

\newcommand{\sA}{\mathscr{A}}

\newcommand{\sD}{\mathscr{D}}
\newcommand{\sE}{\mathscr{E}}

\newcommand{\sL}{\mathscr{L}}

\newcommand{\sN}{\mathscr{N}}

\newcommand{\sP}{\mathscr{P}}
\newcommand{\sQ}{\mathscr{Q}}

\newcommand{\hcL}{{\widehat{\mathcal{L}}}}

\newcommand{\hpsi}{\widehat{\psi}}

\newcommand{\talpha}{\widetilde{\alpha}}

\newcommand{\teta}{\widetilde{\teta}}
\newcommand{\ttheta}{\widetilde{\theta}}

\newcommand{\tphi}{\widetilde{\phi}}
\newcommand{\tvarphi}{\widetilde{\varphi}}

\newcommand{\tchi}{\widetilde{\chi}}

\renewcommand{\=}{ : = }

\DeclareMathOperator{\dist}{dist}

\DeclareMathOperator{\Crit}{Crit} 
\DeclareMathOperator{\Tur}{Tur}

\newcommand{\htop}{h_{\operatorname{top}}}

\DeclareMathOperator{\osc}{osc}
\DeclareMathOperator{\esssup}{ess-sup}
\DeclareMathOperator{\essinf}{ess-inf}
\DeclareMathOperator{\BV}{BV}
\DeclareMathOperator{\Id}{Id}

\DeclareMathOperator{\Var}{Var}

\DeclareMathOperator{\Espace}{E}
\DeclareMathOperator{\Hspace}{H}
\DeclareMathOperator{\Lspace}{L}

\newcommand{\eps}{\varepsilon}

\newcommand{\1}{\pmb{1}}

\begin{document}

\title[Equilibrium states of interval maps]{Equilibrium states of interval maps for hyperbolic potentials}

\author{Huaibin Li}
\address{Huaibin Li, 
Facultad de Matem{\'a}ticas, Pontificia Universidad Cat{\'o}lica de Chile, Avenida Vicu{\~n}a Mackenna~4860, Santiago, Chile}

\email{matlihb@gmail.com}
\thanks{HL was partially supported by the National Natural Science
  Foundation of China (Grant No. 11101124) and FONDECYT grant 3110060
  of Chile. Current address: School of Mathematics and Information Science, Henan University, Kaifeng 475004, China}

\author{Juan Rivera-Letelier}
\thanks{JRL was partially supported by FONDECYT grant 1100922 of Chile.}
\address{Juan Rivera-Letelier, Facultad de Matem{\'a}ticas, Pontificia Universidad Cat{\'o}lica de Chile, Avenida Vicu{\~n}a Mackenna~4860, Santiago, Chile}
\email{riveraletelier@mat.puc.cl}
\urladdr{\url{http://rivera-letelier.org/}}


\begin{abstract}
We study the thermodynamic formalism of sufficiently regular interval maps for H{\"o}lder continuous potentials.
We show that for a hyperbolic potential there is a unique equilibrium state, and that this measure is exponentially mixing.
Moreover, we show the absence of phase transitions: The pressure function is real analytic at such a potential.
\end{abstract}

\maketitle

\section{Introduction}
In this paper we study the thermodynamic formalism of sufficiently regular interval maps for H{\"o}lder continuous potentials.
The case of a piecewise monotone interval map~$f : I \to I$, and a potential~$\varphi : I \to \R$ of bounded variation satisfying
$$ \sup_I \varphi < P(f, \varphi), $$
where~$P(f, \varphi)$ denotes the pressure, is very well understood.
Most results apply under the following weaker condition:
\begin{gather}
\notag
\text{For some integer~$n \ge 1$, the function~$S_n(\varphi) \=
  \sum_{j = 0}^{n - 1} \varphi \circ f^j$ satisfies}
\\
\label{e:hyperbolic potential}
\sup_{I} \frac{1}{n} S_n(\varphi) < P(f, \varphi).
\end{gather}
In what follows, a potential~$\varphi$ satisfying this condition is said to be \emph{hyperbolic for~$f$}.
See for example~\cite{BalKel90,DenKelUrb90,HofKel82b,Kel85,LivSauVai98,Rue94a} and references therein, as well as Baladi's book~\cite[\S$3$]{Bal00b}.
The classical result of Lasota and Yorke~\cite{LasYor73} corresponds to the special case where~$f$ is piecewise~$C^2$ and uniformly expanding, and~$\varphi = - \log |Df|$.

For a complex rational map in one variable~$f$, and a H{\"o}lder continuous potential~$\varphi$ that is hyperbolic for~$f$, a complete description of the thermodynamic formalism was given by Denker, Haydn, Przytycki, and Urba{\'n}ski in~\cite{Hay99,DenPrzUrb96,DenUrb91e,Prz90},\footnote{In this setting, most of the results have been stated for a potential~$\varphi$ satisfying the condition $\sup \varphi < P(f, \varphi)$ that is more restrictive than~$\varphi$ being hyperbolic for~$f$.
General arguments show they also apply to hyperbolic potentials, see~\cite[\S$3$]{InoRiv12}.} extending previous results of Freire, Lopes, and Ma{\~n}{\'e}~\cite{FreLopMan83,Man83}, and Ljubich~\cite{Lju83}.
See also the alternative approach of Szostakiewicz, Urba{\'n}ski, and Zdunik in~\cite{SzoUrbZdu11one}.

In this paper we extend these results to the case of a sufficiently
regular interval map and a H{\"o}lder continuous potential, with the purpose of applying them in the companion paper~\cite{LiRiv14}.
We obtain our main results by constructing a conformal measure with the Patterson-Sullivan method, and then using a result of Keller in~\cite{Kel85}.
This approach is more efficient than the inducing scheme approach
of~\cite{BruTod08}, as it does not rely on any bounded distortion
hypothesis, and it applies to a larger class of maps, including maps
with flat critical points.

We now proceed to describe our main results more precisely.
The class of maps we consider is introduced in~\S\ref{ss:interval maps}, and in~\S\ref{ss:keller spaces} we recall the definition of the function spaces defined by Keller in~\cite{Kel85}.
Our main results are stated in~\S\ref{ss:statements of results}.
\subsection{Interval maps}
\label{ss:interval maps}
Let~$I$ be a compact interval in~$\R$.
A continuous map $f:I\to I$ is \emph{multimodal} if it is not injective, and if there is a finite partition of~$I$ into intervals on each of which $f$ is injective. 

\begin{defi}
Let~$f : I \to I$ be a multimodal map.
The \emph{Julia set~$J(f)$ of~$f$} is the complement of the largest open subset of~$I$ on which the family of iterates of~$f$ is normal.
\end{defi}

In contrast with the complex setting, the Julia set of a multimodal map might be empty, reduced to a single point, or might not be completely invariant.\footnote{This last property can only happen if there is a turning point in the interior of the basin of a one-sided attracting neutral periodic point, that is eventually mapped to this neutral periodic point.}
However, if the Julia set of such a map~$f$ is not completely invariant, then it is possible to make an arbitrarily small smooth perturbation of~$f$ outside a neighborhood of~$J(f)$, so that the Julia set of the perturbed map is completely invariant and coincides with~$J(f)$.
We note also that if~$f$ is differentiable and has no neutral periodic point, then~$J(f)$ is the complement of the basins of periodic attractors.
For background on the theory of Julia sets, see for example~\cite{dMevSt93}.

Given a differentiable map~$f : I \to I$, a point of~$I$ is
\emph{critical for~$f$} if the derivative of~$f$ vanishes at it.
We denote by~$\Crit(f)$ the set of critical points of~$f$.

In what follows we denote by~$\sA$ the collection of all those
differentiable multimodal maps~$f$ such that:
\begin{itemize}
\item
$Df$ is H{\"o}lder continuous;
\item
$\Crit(f)$ is finite;
\item
$J(f)$ contains at least~$2$ points and is completely invariant.
\end{itemize}

\subsection{Keller spaces}
\label{ss:keller spaces}
In this subsection, let~$X$ be  a compact subset of~$\R$, and let~$m$
be an atom-free Borel probability measure on~$X$.
We consider the equivalence relation on the space of complex valued functions defined on~$X$, defined by agreement on a set of full measure with respect to~$m$.
Denote by~$d$ the pseudo-distance on~$X$ defined by
$$ d(x,y)
\=
m(\{z\in X: x\le z\le y \text{   or   } y\le z\le x\}). $$
Note that for all~$x$ in~$X$ and~$\eps>0$, the set
$$ B_d(x,\eps)
\=
\{y\in X: d(x,y)<\eps\} $$
has strictly positive measure with respect to~$m$.

Given a function $h:X \to \C$ and $\eps>0$, for each~$x$ in~$X$ put
$$ \osc(h,\eps, x)
\=
\esssup \left\{ |h(y)-h(y')|: y, y'\in B_d(x,\eps) \right\} $$
and
$$ \osc_1(h,\eps)
\=
\int_{X}\osc(h,\eps,x) \ d m(x). $$
Fix $A>0$, and for each $\alpha$ in $(0,1]$ and each $h:X \to \C$, put $$|h|_{\alpha,1}\=\sup_{\eps\in (0,A]}\frac{\osc_1(h,\eps)}{\eps^{\alpha}} \text{ and } \|h\|_{\alpha,1}\=\|h\|_1+|h|_{\alpha,1}. $$

Note that $|h|_{\alpha,1}$ and $\|h\|_{\alpha,1}$ only depend on the equivalence class of $h$.
Let~$\Hspace^{\alpha,1}(m)$ be the space of equivalence classes of functions~$h:X \to \C$ such that $\|h\|_{\alpha,1}<+\infty$.
Note that $|\cdot|_{\alpha,1}$ and $\|\cdot\|_{\alpha,1}$ induce a semi-norm and a norm on~$\Hspace^{\alpha, 1}(m)$, respectively; by abuse of notation we denote these functions also by $|\cdot|_{\alpha,1}$ and $\|\cdot\|_{\alpha,1}$.
Keller shows in~\cite{Kel85} that~$\Hspace^{\alpha, 1}(m)$ is a Banach space with respect to~$\| \cdot \|_{\alpha, 1}$.
Some properties of these spaces are gathered in~\S\ref{ss:keller spaces properties}.

\subsection{Statement of results}
\label{ss:statements of results}
To state our main results, we recall a few concepts of
thermodynamic formalism, see for example~\cite{Kel98} or~\cite{PrzUrb10} for background.
Let~$(X, \dist)$ be a compact metric space, and let~$T : X \to X$ be a
continuous map.
Denote by~$\cM(X)$ the space of Borel probability measures on~$X$
endowed with the weak* topology, and by~$\cM(X, T)$ the subspace
of~$\cM(X)$ of those measures that are invariant by~$T$.
For each measure~$\nu$ in~$\cM(X, T)$, denote by~$h_{\nu}(T)$ the \emph{measure-theoretic entropy} of~$\nu$.
For a continuous function $\phi: X \to \R$, denote by~$P(X,\phi)$ the \emph{topological pressure of~$T$ for the potential~$\phi$}, defined by
\begin{equation}
\label{e:variational principle}
P(T, \phi)
\=
\sup\left\{h_\nu(T) + \int_X \phi \ d\nu : \nu\in \cM(X, T) \right\}.
\end{equation}
An \emph{equilibrium state of~$T$ for the potential~$\phi$} is a measure at which the supremum above is attained.

For a multimodal map~$f$ and a continuous function~$\varphi :
J(f) \to \R$, we denote~$P(f|_{J(f)}, \varphi)$ just by~$P(f,
\varphi)$.
Moreover, we say~$\varphi$ is \emph{hyperbolic for~$f$}
if~\eqref{e:hyperbolic potential} is satisfied with~$I$ replaced
by~$J(f)$ and~$f$ replaced by~$f|_{J(f)}$.

\begin{defi}
Let~$f$ be a multimodal map.
Given a Borel measurable function $g: J(f) \to [0,+\infty)$, a Borel probability measure~$\mu$ on~$J(f)$ is \emph{$g$-conformal for~$f$}, if for each Borel set~$A$ on which~$f$ is injective, we have
\begin{equation*}
\label{e:confor}
\mu(f(A))=\int_A g \ d\mu.
\end{equation*}
\end{defi}

\begin{theoalph}
\label{t:conformal measure}
Let~$f$ be an interval map in~$\sA$.
Then for every H{\"o}lder continuous potential $\varphi:J(f) \to \R$
that is hyperbolic for~$f$ there is an atom-free  $\exp(P(f,\varphi)-\varphi)$-conformal measure for $f$.
If in addition~$f$ is topologically exact on~$J(f)$, then the support of~$\mu$ is equal to~$J(f)$.
\end{theoalph}

Given a multimodal map~$f$, and a continuous potential
$\varphi:J(f)\rightarrow \R$, denote by~$\cL_\varphi$ the
\emph{transfer} or \emph{Ruelle-Perron-Frobenius operator}, acting on
the space of bounded functions defined on~$J(f)$ and taking values
in~$\C$, defined as follows
$$ \cL_{\varphi}(\psi)(x)
\=
\sum_{y\in f^{-1}(x)}\exp\left(\varphi(y)\right) \psi(y). $$

The following is our main result.
It is obtained by combining Theorem~\ref{t:conformal measure}, a result of Keller in~\cite{Kel85}, and known arguments.

\begin{theoalph}
\label{t:spectral gap}
Let~$f$ be an interval map in~$\sA$ that is topologically exact on~$J(f)$.
Let~$\alpha$ be in~$(0, 1]$, and let~$\varphi : J(f) \to \R$ be a
H{\"o}lder continuous potential of exponent~$\alpha$ that is
hyperbolic for~$f$.
Finally, let~$\mu$ be an atom-free  $\exp(P(f,\varphi)-\varphi)$-conformal measure for $f$ given by Theorem~\ref{t:conformal measure}.
Then there is~$A > 0$ such that for the space~$\Hspace^{\alpha, 1}(\mu)$ defined in~\S\ref{ss:keller spaces} with~$X = J(f)$, the following properties hold.
\begin{description}
\item[Spectral gap]
The operator~$\cL_{\varphi}$ maps~$\Hspace^{\alpha,1}(\mu)$ to itself, and~$\cL_{\varphi}|_{\Hspace^{\alpha, 1}(\mu)}$ is bounded.
Moreover, the number $\exp(P(f, \varphi))$ is an eigenvalue of algebraic multiplicity~$1$ of~$\cL_{\varphi}|_{\Hspace^{\alpha, 1}(\mu)}$, and there exists~$\rho$ in~$(0, \exp(P(f, \varphi)))$ such that the spectrum of~$\cL_{\varphi}|_{\Hspace^{\alpha,1}(\mu)}$ is contained in~$B(0, \rho) \cup \{ \exp(P(f, \varphi)) \}$.
\item[Equilibrium state]
There is a unique equilibrium state~$\nu$ of~$f$ for the potential~$\varphi$.
Moreover, this measure is absolutely continuous with respect to~$\mu$, and the measure-theoretic entropy of~$\nu$ is strictly positive.
Finally, there is a constant $C > 0$ such that for every integer~$n \ge 1$, every bounded measurable function~$\phi : J(f) \to \C$, and every~$\psi$ in~$\Hspace^{\alpha,1}(\mu)$, we have
$$ C_n(\phi, \psi)
\=
\left|\int_{J(f)}\phi \circ f^n \cdot \psi \ d\nu - \int_{J(f)}\phi \ d\nu \int_{J(f)} \psi \ d\nu \right|
\le
C \|\phi\|_{\infty} \| \psi \|_{\alpha, 1}  \rho^n. $$
\item[Real analyticity of pressure]
For each H{\"o}lder continuous function~$\chi : J(f) \to \R$, the function~$t \mapsto P(f, \varphi + t \chi)$ is real analytic on a neighborhood of~$t = 0$.
\end{description}
\end{theoalph}

\begin{rema}
It is well-known that the spectral gap property implies that the equilibrium state has strong stochastic properties.
In Theorem~\ref{t:spectral gap} we emphasize the exponential decay of correlations.
See~\cite[Theorem~$3.3$]{Kel85} for the Central Limit Theorem and the Almost Sure Invariance Principle.
\end{rema}

Given a compact subset~$X$ of~$\R$ and~$\alpha$ in~$(0,1]$, for each H{\"o}lder continuous function $h : X \to \C$ of exponent~$\alpha$, put
$$ |h|_\alpha \= \sup_{x,x'\in X,x\neq x'} \frac{|h(x)-h(x')|}{|x-x'|^\alpha}
\text{  and  }
\|h\|_\alpha \= \|h\|_{\infty} + |h|_\alpha. $$
Then for each~$A > 0$ and each atom-free Borel probability measure~$m$
on~$X$, each H{\"o}lder continuous function $h : X \to \C$ of exponent~$\alpha$ is in~$\Hspace^{\alpha, 1}(m)$ and we have
$$ \| h \|_{\alpha, 1}
\le
2^{\alpha} \max \{ 1, (\sup X - \inf X)^\alpha \} \| h \|_{\alpha}, $$
see~\eqref{e:1} and part~$2$ of Proposition~\ref{p:keller spaces properties}.
Thus, the following corollary is a direct consequence of Theorem~\ref{t:spectral gap}.
\begin{coro}
\label{c:stochastic equilibria}
Let~$I$ be a compact interval of $\R$ and let~$f : I \to I$ be an interval map in~$\sA$ that is topologically exact on~$I$.
Then for every H{\"o}lder continuous potential~$\varphi : I \to \R$ that is hyperbolic for $f$, there is a unique equilibrium state~$\nu$ of~$f$ for the potential~$\varphi$ and the measure-theoretic entropy of this measure is strictly positive.
Moreover, if~$\alpha$ in~$(0, 1]$ is such that~$\varphi$ is H{\"o}lder continuous of exponent~$\alpha$, then there are constants $C > 0$ and~$\rho$ in~$(0, 1)$, such that for every integer~$n \ge 1$, every bounded measurable function~$\phi : J(f) \to \C$, and every H{\"o}lder continuous function $\psi : J(f) \to \C$ of exponent~$\alpha$, we have
$$ C_n(\phi, \psi)
\le
C \|\phi\|_{\infty} \| \psi \|_{\alpha}  \rho^n. $$
Finally, for every H{\"o}lder continuous function~$\chi : I \to \R$, the function~$t \mapsto P(\varphi + t \chi)$ is real analytic on a neighborhood of~$t = 0$.
\end{coro}

In~\cite[Theorem~A]{LiRiv14} we show that for a map~$f$ as in
Corollary~\ref{c:stochastic equilibria}, every H{\"o}lder continuous
potential~$\varphi$ is hyperbolic for~$f$, provided that~$f$ satisfies some
additional regularity assumptions, that all the periodic points of~$f$ are hyperbolic repelling, and that for every critical value~$v$ of~$f$ we have
$$
\lim_{n \to + \infty} |Df^n(v)| = + \infty.
$$
Recall that a periodic point~$p$ of~$f$ of period~$n$ is \emph{hyperbolic repelling}, if~$|D f^n(p)| > 1$.
Thus, for such a map~$f$ the conclusions of Corollary~\ref{c:stochastic equilibria} hold for all H{\"o}lder continuous functions~$\varphi$, $\psi$, and~$\chi$.
See~\cite[Main Theorem]{LiRiv14} for a more general formulation of this result.

For a H{\"o}lder continuous potential, Corollary~\ref{c:stochastic
  equilibria} improves~\cite[Theorem~$4$]{BruTod08} in various ways.
The first is that in Corollary~\ref{c:stochastic equilibria} no bounded distortion hypothesis is assumed, in contrast with~\cite[Theorem~$4$]{BruTod08} where the existence of an induced map with bounded distortion is assumed. 
The second is that the hypothesis that the potential is hyperbolic in Corollary~\ref{c:stochastic equilibria} is weaker than the ``bounded range'' condition assumed in~\cite[Theorem~$4$]{BruTod08}.
See Appendix~\ref{s:big oscillation} for the definition of the bounded range condition and for examples showing that this condition is more restrictive than hyperbolicity.
The fact that our results hold for hyperbolic potentials, and not only
for potentials satisfying the bounded range condition, is crucial to
obtain the main results of the companion paper~\cite{LiRiv14}.
Finally, Corollary~\ref{c:stochastic equilibria} applies to a larger class of
interval maps, including maps having flat critical points.

\subsection{Organization}
\label{ss:organization}
The proof of Theorem~\ref{t:conformal measure} occupies~\S\ref{s:tree pressure} and~\S\ref{s:conformal measure}.
In~\S\ref{s:tree pressure} we show that for a hyperbolic potential the pressure function can be calculated using iterated preimages of any given point (Corollary~\ref{c:tree pressure}).
Then this is used in~\S\ref{s:conformal measure} to apply the Patterson-Sullivan method to construct a conformal measure (Proposition~\ref{p:conformal measure}).

In~\S\ref{sec:keller} we first recall some properties of the function spaces defined by Keller in~\S\ref{ss:keller spaces properties}.
In~\S\ref{ss:quasicompactness} we recall the main results of~\cite{Kel85} (Theorem~\ref{t:quasicompactness}) and deduce from it some known consequences needed for the proof of Theorem~\ref{t:spectral gap} (Corollary~\ref{c:quasicompactness}).
The proof of Theorem~\ref{t:spectral gap} is given in~\S\ref{s:proof of main theorem}.

\subsection{Acknowledgments}
We would like to thank Michal Szostakiewicz for a useful discussion related to Lemma~\ref{l:big oscillation}.
The second named author would also like to thank Daniel Smania for quasi-periodic discussions about Keller spaces.
Finally, we would like to thank the referees for their remarks that
helped improve the exposition of the paper.
\section{Iterated preimages pressure}
\label{s:tree pressure}
The main goal of this section is to prove the following proposition.
For a multimodal map~$f$ and a continuous potential~$\varphi : J(f) \to
\R$, put
$$ \hcL_{\varphi}
\=
\exp(- P(f, \varphi)) \cL_{\varphi}. $$
\begin{prop}
\label{p:tree pressure}
Let~$f$ be an interval map in~$\sA$ that is topologically exact on its Julia set~$J(f)$, and denote by~$\1$ the function defined on~$J(f)$ that is constant equal to~$1$.
Then for every $\eps>0$, and every H{\"o}lder continuous potential $\varphi:J(f)\to \R$ that is hyperbolic for~$f$, we have for every sufficiently large integer~$n$
$$\exp(-\eps n)
\le
\inf_{J(f)}\hcL_{\varphi}^n(\1)
\le
\sup_{J(f)}\hcL_{\varphi}^n(\1)
\le
\exp(\eps n). $$
\end{prop}

The following corollary is a direct consequence of the proposition.
This corollary is used in the next section.
\begin{coro}
\label{c:tree pressure}
Let~$f$ be an interval map in~$\sA$ that is topologically exact on~$J(f)$.
Then for every H{\"o}lder continuous potential $\varphi:J(f)\to \R$ that is hyperbolic for $f$, and every point~$x_0$ in $J(f)$, we have $$P(f,\varphi)=\lim_{n \to + \infty} \frac{1}{n} \log \sum_{y\in f^{-n}(x_0)} \exp(S_n(\varphi)(y)). $$
\end{coro}

The remainder of this section is devoted to the proof of
Proposition~\ref{p:tree pressure}, which depends on several lemmas. 

\begin{lemm}
\label{l:onetime}
Let~$f$ be a Lipschitz multimodal map and let $\varphi:J(f)\to \R$ be a H{\"o}lder continuous potential. 
Then for every integer $N\ge 1$, there is a constant $C>1$ such that the function~$\tvarphi\= \frac{1}{N}S_N(\varphi)$ satisfies the following properties:
\begin{enumerate}
\item[1.]
The function~$\tvarphi$ is H{\"o}lder continuous of the same exponent as $\varphi$, $P(f, \tvarphi) = P(f, \varphi)$, and~$\varphi$ and~$\tvarphi$ share the same equilibrium states;
\item[2.]
For every integer $n\ge 1$, we have
$$ \sup_{J(f)} \left| S_n(\tvarphi) - S_n(\varphi) \right|
\le
C. $$
\end{enumerate}
\end{lemm}
\begin{proof}
Let~$h : J(f) \to \R$ be defined by
$$ h \= - \frac{1}{N} \sum_{j = 0}^{N - 1} (N- 1 -j) \varphi \circ f^j, $$
and note that $\varphi = \tvarphi + h - h \circ f$.

\partn{1}
Since~$f$ is Lipschitz, $\tvarphi$ is H{\"o}lder continuous of the same exponent as~$\varphi$.
On the other hand, for every invariant measure~$\nu$ of~$f$, we have
$$ \int_{J(f)} \tvarphi \ d \nu
=
\int_{J(f)} (\varphi+h\circ f-h) \ d \nu
=
\int_{J(f)} \varphi \ d \nu, $$
so~$P(f, \tvarphi) = P(f, \varphi)$, and~$\varphi$ and~$\tvarphi$ share the same equilibrium states.

\partn{2}
Note that for every $n\ge N$,  we have
\begin{align*}
S_n(\tvarphi)
=
S_n(\varphi+h\circ f-h)
=
S_n(\varphi)+h\circ f^n-h.
\end{align*}
This implies the desired inequality with $C=(N-1)\left( \sup_{J(f)}\varphi
  -\inf_{J(f)}\varphi \right)$.
\end{proof}

Given an integer~$n \ge 1$ and a point~$x$ in the domain of~$f$, a preimage~$y$ of~$x$ by~$f^n$ is \emph{critical} if~$Df^n(y) = 0$, and it is \emph{non-critical} otherwise.
\begin{lemm}
\label{l:diffeomorphic pressure}
Let~$f$ be an interval map in~$\sA$ that is topologically exact on~$J(f)$, and let~$\tvarphi : J(f) \to \R$ be a H{\"o}lder continuous potential satisfying~$\sup_{J(f)} \tvarphi < P(f, \tvarphi)$.
Then for every~$\varepsilon > 0$ and every point~$x_0$ of~$J(f)$ having infinitely many non-critical preimages, there is~$\delta > 0$ such that the following property holds: If for each integer~$n \ge 1$ we denote by~$\fD_n$ the collection of diffeomorphic pull-backs of~$B(x_0, \delta)$ by~$f^n$, then
$$ \liminf_{n \to + \infty} \frac{1}{n} \log \sum_{W \in \fD_n}
\inf_{W \cap J(f)} \exp \left( S_n(\tvarphi) \right)
\ge
P(f, \tvarphi) - \varepsilon. $$
\end{lemm}

The proof of this lemma is based on Pesin's theory, as adapted to
interval maps in~$\sA$ by Dobbs in~\cite[Theorem~$6$]{Dob0801}, and Katok's
theory, as adapted to one-dimensional maps by
Przytycki and Urba{\'n}ski in~\cite[\S$11.6$]{PrzUrb10}.
The proofs in~\cite[\S$11.6$]{PrzUrb10} are written for complex
rational maps, but they apply without change to interval maps in~$\sA$
by using~\cite[Theorem~$6$]{Dob0801} instead of~\cite[Corollary~$11.2.4$]{PrzUrb10}.

\begin{proof}
In part~$1$ we show that there is~$\delta_0 > 0$ and a forward invariant compact set~$X$ on
which~$f$ is uniformly expanding, such that the desired assertion
holds for every point~$x_0$ in~$X$ with~$\delta = \delta_0$.
In part~$2$ we deal with the general case using this special case.

\partn{1}
Since by assumption~$\sup_{J(f)} \tvarphi < P(f, \tvarphi)$, there
is~$\varepsilon > 0$ so that~$\varepsilon < P(f, \tvarphi) - \sup_{J(f)}
\tvarphi$.
Let~$\nu$ be a measure in~$\cM(J(f), f)$ such that
$$ h_\nu(f) + \int_{J(f)} \tvarphi \ d\nu
\ge
P(f, \tvarphi) - \varepsilon > \sup_{J(f)} \tvarphi. $$
Replacing~$\nu$ by one of its ergodic components if necessary,
assume~$\nu$ is ergodic.
We thus have
$$ h_\nu(f)
>
\sup_{J(f)} \tvarphi - \int_{J(f)} \tvarphi \ d \nu
\ge
0, $$
and then Ruelle's inequality implies that the Lyapunov exponent
of~$\nu$ is strictly positive, see~\cite{Rue78}.
By~\cite[Theorem~$11.6.1$]{PrzUrb10} there is a compact and forward invariant subset~$X$ of~$J(f)$ on which~$f$ is topologically transitive, so that~$f : X \to X$ is open and uniformly expanding, and so that
$$ P(f|_X, \tvarphi|_X) \ge P(f, \tvarphi) - \varepsilon. $$
It follows that there is~$\delta_0 > 0$ such that the desired property holds for
every point~$x_0$ in~$X$ with~$\delta = \delta_0$, see for example~\cite[Proposition~$4.4.3$]{PrzUrb10}.

\partn{2}
The hypothesis that~$x_0$ has infinitely many non-critical preimages implies that there is a non-critical preimage~$x_0'$ of~$x_0$ such that all preimages of~$x_0'$ are non-critical.
Since~$f$ is topologically exact on~$J(f)$, there is a preimage
of~$x_0'$ in~$B(X, \delta_0)$, and therefore there is an integer~$n
\ge 1$ and a non-critical preimage~$x_0''$ of~$x_0$ by~$f^n$ in~$B(X,
\delta_0)$ all whose preimages are non-critical.
It follows that there is~$\delta > 0$ such that the pull-back of~$B(x_0, \delta)$ by~$f^n$ that contains~$x_0''$ is contained in~$B(X, \delta_0)$.
Then the desired assertion follows from part~$1$.
\end{proof}

\begin{lemm}
\label{l:obelow}
Let~$f$ be an interval map in~$\sA$ that is topologically exact on~$J(f)$, and let $\varphi:J(f)\to \R$ be a H{\"o}lder continuous potential that is hyperbolic for~$f$.
Then for every $\eps>0$ there is $N_0>0$ such that for every integer~$n\ge N_0$, we have
$$ \inf_{J(f)}\hcL_\varphi^n(\1)
\ge
\exp(-\eps n). $$
\end{lemm}
\begin{proof}
Let $C>1$ be the constant given by Lemma~\ref{l:onetime}. 
Since~$\varphi$ is hyperbolic for~$f$, there is an integer~$N \ge 1$ such that the function~$\tvarphi \= \frac{1}{N} S_N(\varphi)$ satisfies~$\sup_{J(f)} \tvarphi < P(f, \varphi)$.
By part~$1$ of Lemma~\ref{l:onetime}, the function~$\tvarphi$ is
H{\"o}lder continuous and
$$ P(f, \tvarphi)
=
P(f, \varphi)
>
\sup_{J(f)}\tvarphi. $$
In view of part~$2$ of Lemma~~\ref{l:onetime}, to complete the proof of the lemma it suffices to prove that for every $\eps>0$ there is $N_0>0$ such that for every $n\ge N_0$ we have
$$ \inf_{J(f)}\hcL_{\tvarphi}^n(\1)\ge \exp(C)\exp(-\eps n). $$
Let~$x_0$ be a point of $J(f)$ that has infinitely many non-critical preimages. 
Let~$\delta>0$ and for each integer~$n \ge 1$, let~$\fD_n$ be as in
Lemma~\ref{l:diffeomorphic pressure} with~$\varepsilon$ replaced by~$\varepsilon/2$.
Then there is $n_0\ge 1$ such that for every integer $n\ge n_0$ we have
$$ \frac{1}{n} \log \sum_{W \in \fD_n} \inf_{W \cap J(f)} \exp \left(
  S_n(\tvarphi) \right)
\ge
P(f, \tvarphi)-\eps/2. $$
This implies for each~$n\ge n_0$ and  every~$x^*$ in~$B(x_0, \delta)\cap J(f)$, we have
\begin{equation}
\label{e:nb}
\cL_{\tvarphi}^n(\1)(x^*)
\ge
\exp \left( n \left( P(f,\tvarphi) - \eps/2 \right) \right).
\end{equation}
On the other hand, since $f$ is topologically exact on $J(f)$, there is $n_1\ge 1$ such that~$f^{n_1}(B(x_0,\delta))\supset J(f)$.
Given~$x$ in~$J(f)$, let~$x'$ in~$B(x_0, \delta)\cap J(f)$ be such that $f^{n_1}(x')=x$.
It follows that for every $n\ge n_1+n_0$
\begin{align*}
\cL_{\tvarphi}^n(\1)(x)
& =
\sum_{y\in f^{-n}(x)}\exp\left(S_n(\tvarphi)(y)\right)
\\ & =
\sum_{y'\in f^{-n_1}(x)}\sum_{y\in
  f^{-(n-n_1)}(y')}\exp\left(S_{n-n_1}(\tvarphi)(y)+S_{n_1}(\tvarphi)(y')\right)
\\ & \ge
\exp \left( S_{n_1}(\tvarphi)(x') \right) \cL_{\tvarphi}^{n-n_1}(\1)(x')
\\ & \ge
\exp\left(n_1\inf_{J(f)}\tvarphi\right) \cL_{\tvarphi}^{n-n_1}(\1)(x')
\\ & \ge
\exp\left(n_1\inf_{J(f)}\tvarphi\right)
\exp \left((n-n_1) \left( P(f,\tvarphi) - \eps/2 \right) \right).
\end{align*}
Let $N_@\ge 0$ be such that
$$ \exp\left(n_1 \inf_{J(f)}\tvarphi\right)
\ge
\exp(C) \exp\left(n_1 P(f,\tvarphi)-(\eps N_@)/2\right). $$
Then for every~$x$ in~$J(f)$ and every integer $n \ge \max \{ n_1 + n_0, N_@ \}$, we have
\begin{equation*}
\begin{split}
\cL_{\tvarphi}^n(\1)(x)
& \ge
\exp(C) \exp\left(n_1 P(f,\tvarphi)-(\eps N_@)/2 + (n-n_1) \left( P(f,\tvarphi) - \eps/2 \right) \right)
\\ & \ge
\exp(C)\exp\left(n \left( P(f,\tvarphi)-\eps \right) + (\varepsilon /
  2)(n + n_1 - N_@) \right)
\\ & \ge
\exp(C)\exp\left(n \left( P(f,\tvarphi)-\eps \right)\right).
\end{split}
\end{equation*}
This proves the desired inequality with $N_0 = \max \{n_1 + n_0, N_@ \}$, and so the proof of the lemma is complete.
\end{proof}

In order to complete the proof of Proposition~\ref{p:tree pressure},  we recall the definition of topological pressure using ``$(n,\varepsilon)$\nobreakdash-separated sets.''
Let $(X, \dist)$ be a compact metric space and let $T: X \to X$ be a continuous map.
For each integer~$n \ge 1$ the function~$\dist_n : X \times X \to \R$ defined by
$$ \dist_n(x,y)
\=
\max \left\{ \dist(T^i(x),T^i(y)) : i \in \{ 0,1,\cdots, n-1 \} \right\}, $$
is a metric on~$X$.
Note that $\dist_1 = \dist$.
For~$\varepsilon > 0$, and an integer~$n \ge 1$, a pair of points~$x$ and~$x'$ of~$X$ are \emph{$(n,\varepsilon)$\nobreakdash-close} if $\dist_n(x,y)< \varepsilon$.
Moreover, a subset~$F$ of~$X$ is \emph{$(n,\varepsilon)$\nobreakdash-separated}, if it does not contain~$2$ distinct points that are~$(n, \varepsilon)$\nobreakdash-close.
For a continuous function~$\phi : X \to \R$, the pressure~$P(T, \phi)$, defined in~\eqref{e:variational principle}, satisfies
\begin{equation}
\label{e:sep}
 P(T,\phi)
=
\lim_{\varepsilon\to 0}\limsup_{n\to +\infty}\frac{1}{n}\log \sup_F \left(\sum_{y\in F} \exp(S_n(\phi(y)))\right),
\end{equation}
where the supremum is taken over all
$(n,\varepsilon)$\nobreakdash-separated subsets~$F$ of~$X$, see for
example~\cite{Kel98} or~\cite{PrzUrb10}.

\begin{proof}[Proof of Proposition~\ref{p:tree pressure}]
Put~$\Crit'(f) \= \Crit(f) \cap J(f)$.

In view of Lemma~\ref{l:obelow} it is enough to show the following: \emph{For every $\eps>0$ there is $N_1>0$ such that for every integer $n\ge N_1$, we have}
$$ \sup_{J(f)}\hcL_\varphi^n(\1) \le \exp\left(\eps n\right). $$
The proof of this fact can be adapted from~\cite[Lemma~$4$]{Prz90}.
The details are as follow.
By~\eqref{e:sep} there is~$\eps_0>0$ such that for every~$\eps_*$ in $(0, \eps_0)$ there is an integer $N(\eps_*)\ge 1$ such that for every $n\ge N(\eps_*)$ and every $(n,\varepsilon_*)$\nobreakdash-separated subset~$\sN$ of~$J(f)$, we have
 \begin{equation}
\label{e:upp}
 \sum_{y\in \sN} \exp \left(S_n(\varphi)(y)\right)\le \exp\left(n(P(f,\varphi)+\eps/2)\right).
 \end{equation}

Put~$N\= \# \Crit(f) + 1$, and let~$L \ge 2 \# \Crit(f)$ be large enough so
that
$$ N^{2(\#\Crit'(f))^2 / L}
\le
\exp\left(\eps/2\right). $$
Note that a point in the domain of~$f$ can have at most~$N$ preimages by~$f$.
On the other hand, there is~$\varepsilon'$ in $(0,\eps_*)$ such that
for every~$c$ in~$\Crit'(f)$, every~$x$ in~$B(c, 2\varepsilon')$, and every~$j$ in~$\{1,2,\ldots, L\}$, we have either
$$ f^j(x) \not \in  B(\Crit(f), 2 \varepsilon')
\text{ or }
f^j(c) \in \Crit(f). $$
Since no critical point of~$f$ in~$J(f)$ is periodic, for each critical point~$c$ of~$f$ in~$J(f)$ there are at most~$\# \Crit'(f) - 1$ integers~$j \ge 1$ such that~$f^j(c)$ is in~$\Crit(f)$.
Reducing~$\varepsilon'$ if necessary, assume that for every~$x$
in~$J(f)$  satisfying~$\dist (x, \Crit'(f))\ge 2\varepsilon'$ the map~$f$ is injective on~$B(x, \varepsilon')$.

Given an integer $n \ge 1$ and a point~$x$ of~$J(f)$, denote by~$P_{(n, \varepsilon')}(x)$ the number of points in $f^{-n}(f^n(x))$ that are $(n, \varepsilon')$\nobreakdash-close to~$x$.
Note that~$P_{(n, \varepsilon')}(x) \ge 1$ and put
$$ P_{(n, \varepsilon')}\=\sup_{x \in J(f)}P_{(n, \varepsilon')}(x). $$
Let~$x_0$ be a point of~$J(f)$. 
Then for every integer $n \ge 1$, the set $f^{-n}(x_0)$ can be partitioned into~$P_{(n, \varepsilon')}$ sets, each of which is $(n, \varepsilon')$\nobreakdash-separated.
So there is a $(n, \varepsilon')$\nobreakdash-separated subset~$\sN$ of~$f^{-n}(x_0)$ such that
$$ \sum_{y\in \sN} \exp \left(S_n(\varphi)(y)\right)
\ge
\frac{1}{P_{(n, \varepsilon')}}\sum_{y\in f^{-n}(x_0)} \exp \left(S_n(\varphi)(y)\right). $$ 
Together with~\eqref{e:upp}, for every $n\ge N(\eps')$ we have $$\sum_{y\in f^{-n}(x_0)} \exp \left(S_n(\varphi)(y)\right)\le P_{(n, \varepsilon')}\exp\left(n(P(f,\varphi)+\eps/2)\right). $$
Thus, to complete the proof of the proposition it suffices to prove there is an integer $N'\ge 1$ such that for every $n\ge N'$ we have $P_{(n, \varepsilon')}\le \exp(n\eps/2). $

Fix a point~$x_0$ in~$J(f)$, and for each point~$x$ of~$J(f)$ put $P_{(0, \varepsilon')}(x)=1$.
Note that if for some $n\ge 2$ a point~$y$ in~$J(f)$ and a point~$y'$ in~$f^{-n}(f^n(x_0))$ are $(n, \varepsilon')$\nobreakdash-close, then~$f(y)$ and~$f(y')$ are~$(n - 1, \varepsilon')$\nobreakdash-close.
Therefore we have
$$ P_{(n, \varepsilon')}(x_0)
\le
N \cdot P_{(n-1, \varepsilon')}(f(x_0)), $$
and when $f$ is injective on~$B(x_0, \varepsilon')$, we have $P_{(n, \varepsilon')}(x_0)\le P_{(n-1, \varepsilon')}(f(x_0))$.
In particular,  when $\dist (x_0, \Crit'(f))\ge 2\varepsilon'$ we have $P_{(n, \varepsilon')}(x_0)\le P_{(n-1, \varepsilon')}(f(x_0))$.
By induction and the definition of~$\varepsilon'$ we obtain
$$ P_{(n, \varepsilon')}(x_0)
\le
N^{\# \Crit'(f) \left( \frac{\# \Crit'(f)}{L} n + \# \Crit'(f)
  \right)}
\le
N^{2 n (\# \Crit'(f))^2 /L}. $$
Since~$x_0$ is an arbitrary point in~$J(f)$, for every $n\ge L$ we have
$$ P_{(n, \varepsilon')}
\le
N^{2n(\#\Crit'(f))^2/L}
\le
\exp\left(\eps n/2\right). $$
This completes the proof of the proposition.
\end{proof}

\section{Conformal measure}
\label{s:conformal measure}
The main goal of this section is to prove Theorem~\ref{t:conformal
  measure}, on the existence of a conformal measure.
We use a general method of construction conformal measures, usually known as the ``Patterson-Sullivan method''.
For rational maps, this method was introduced by Sullivan in~\cite{Sul83}, see also~\cite{DenUrb91f} and~\cite[\S$12$]{PrzUrb10}.

We proceed to describe a preliminary fact needed in the construction.
Given a sequence $(a_n)_{n=1}^{+\infty}$ of real numbers such that
$$c \= \limsup_{n\to +\infty} \frac{a_n}{n} < +\infty, $$
the number $c$ is called the \emph{transition parameter of $(a_n)_{n=1}^{+\infty}$}. 
It is uniquely determined by the property that the series
$$ \sum_{n=1}^{+\infty}\exp(a_n-ns) $$
converges for $s>c$ and diverges for $s<c$. 
For $s=c$ the sum may converge or diverge.

For a proof of the following simple fact, see for example~\cite[Lemma~$12.1.2$]{PrzUrb10}.
\begin{lemm}
\label{t:transition parameter}
Let $(a_n)_{n=1}^{+\infty}$ be a sequence of real numbers having transition parameter~$c$.
Then there is a sequence  $(b_n)_{n=1}^{+\infty}$ of positive real numbers such that
$$ \sum_{n=1}^{+\infty}b_n\exp(a_n-ns)
\begin{cases}
< +\infty & \text{if } s>c \\
= +\infty & \text{if } s\le c,
\end{cases} $$
and $\lim_{n\to +\infty}b_n/b_{n+1}=1$.
\end{lemm}

In view of Corollary~\ref{c:tree pressure},
Theorem~\ref{t:conformal measure} is a direct consequence of the following proposition. 
A \emph{turning point} of a multimodal map~$f : I \to I$ is a point
in~$I$ at which~$f$ is not locally injective.

\begin{prop}
\label{p:conformal measure}
Let~$f : I \to I$ be a multimodal map, let~$X$ be a compact subset
of~$I$ that contains at least~$2$ points and satisfies~$f^{-1}(X)
\subset X$, and let $\varphi: X \to \R$ be a continuous
function. 
Assume that~$f$ has no periodic turning point in~$X$, and that there
is a point~$x_0$ of~$X$ and an integer~$N\ge 1$ such that the number
$$ P_{x_0}
\=
\limsup_{n \to + \infty} \frac{1}{n} \log \sum_{y\in f^{-n}(x_0)} \exp(S_n(\varphi)(y)) $$
satisfies~$P_{x_0}> \sup_X \frac{1}{N} S_N(\varphi)$.
Then there is an atom-free $\exp(P_{x_0}-\varphi)$-conformal measure
for~$f$.
If in addition~$f$ is topologically exact on~$X$, then the support of this conformal measure is equal to~$X$.
\end{prop}

\begin{proof}
In part~$1$ below  we construct a measure that is conformal outside
$\Tur(f) \cup \partial I$.
In parts~$2$ and~$3$ we prove that this measure has no atoms, and in part~$4$ we conclude the proof of the proposition.

Denote by~$\Tur(f)$ the set of turning points of~$f$.

\partn{1} \emph{Construction of the measure~$\mu$.} For each integer $n\ge 1$ put
$$ a_n \= \log \sum_{y\in f^{-n}(x_0)} \exp(S_n(\varphi)(y)). $$
The hypotheses imply that~$P_{x_0}$ is the transition parameter of~$(a_n)_{n=1}^{+\infty}$. 
Let~$(b_n)_{n=1}^{+\infty}$ be given by Lemma~\ref{t:transition parameter}, and for each real number~$s> P_{x_0}$ define
$$ M_s
\=
\sum_{n=1}^{+\infty}b_n\exp\left(a_n-ns\right), $$
and the probability  measure
\begin{equation}
\label{e:normalizedmeasure}
m_s\=\frac{1}{M_s}\sum_{n=1}^{+\infty}\sum_{y\in f^{-n}(x_0)}b_n\exp\left(S_n(\varphi)(y)-ns\right)\delta_y.
\end{equation}
Let~$\mu$ be a weak$^*$ accumulation point, as $s\to P_{x_0}^{+}$, of the measures $\{m_s: s>P_{x_0}\}$ defined by~\eqref{e:normalizedmeasure}. 
Since~$X$ is compact and~$f^{-1}(X) \subset X$, the support of~$\mu$
is contained in~$X$.

Let~$A$ be a Borel subset of~$I$ on which~$f$ is injective. 
Using~\eqref{e:normalizedmeasure} and $f^{-1}(f^{-n}(x_0))=f^{-(n+1)}(x_0)$, for each integer $n\ge 1$ we have
\begin{align*}
m_s(f(A))
&=
\frac{1}{M_s}\sum_{n=1}^{+\infty}\sum_{y\in f^{-n}(x_0)\cap f(A)}b_n\exp\left(S_n(\varphi)(y)-ns\right)
\\ & =
\frac{1}{M_s}\sum_{n=1}^{+\infty}\sum_{z\in A\cap f^{-(n+1)}(x_0)}b_n\exp\left(S_n(\varphi)(f(z))-ns\right).
\end{align*}
It follows that
\begin{align*}
\Delta_A(s)
& \=
\left|m_s(f(A))-\int_A \exp\left(P_{x_0}-\varphi\right) \ d m_s \right|
\\ & =
\left|\frac{1}{M_s}\sum_{n=1}^{+\infty}\sum_{y\in A\cap f^{-(n+1)}(x_0)}b_n\exp\left(S_{n}(\varphi)(f(y))-ns\right)
\right. \\ & \quad \left.
-\frac{1}{M_s}\sum_{n=1}^{+\infty}\sum_{y\in A\cap f^{-n}(x_0)} b_n\exp\left(S_{n}(\varphi)(y)-ns\right)\exp\left(P_{x_0}-\varphi(y)\right)
\right|
\\ & =
\frac{1}{M_s}\left|\sum_{n=1}^{+\infty}\left(\left(b_n-b_{n+1}\exp\left(P_{x_0}-s\right)\right)\sum_{y\in A\cap f^{-(n+1)}(x_0)}\exp\left(S_{n}(\varphi)(f(y))-ns\right)\right)
\right. \\ & \quad \left.
- b_1\sum_{y\in A\cap f^{-1}(x_0)}\exp\left(P_{x_0}-s\right) \right|
\\ & \le
\frac{1}{M_s}\sum_{n=1}^{+\infty}\left(\left|1-\frac{b_{n+1}}{b_{n}} \exp\left(P_{x_0}-s\right)\right| b_{n}\sum_{y'\in f^{-n}(x_0)}\sum_{y\in A\cap f^{-1}(y')} \exp\left(S_{n}(\varphi)(y')-ns\right)\right)
\\ & \quad
+ \frac{1}{M_s}b_1\exp\left(P_{x_0}-s\right)\left(\#(A\cap f^{-1}(x_0))\right).
\end{align*}
Thus, if we put $K\=\sup_{x\in X}\# \left(f^{-1}(x)\right)$ and
\begin{multline*}
\Delta(s)
\=
\frac{K}{M_s}\sum_{n=1}^{+\infty}\left(
  \left|1-\frac{b_{n+1}}{b_{n}}\exp(P_{x_0}-s)\right|
  b_{n}\sum_{y'\in f^{-n}(x_0)} \exp \left( S_{n}(\varphi)(y')-ns
  \right) \right)
\\
+ \frac{K }{M_s}b_1 \exp\left(P_{x_0}-s\right),
\end{multline*}
then for every Borel subset~$A$ of~$I$ on which~$f$ is injective we have
\begin{equation*}
\Delta_A(s)\le \Delta(s).
\end{equation*}
On the other hand, by Lemma~\ref{t:transition parameter} and our hypotheses,  we know that
$$ \lim_{n\to +\infty} b_n/b_{n+1} =1
\text{ and }
\lim_{s\to P_{x_0}^{+}}M_s = +\infty, $$
so we obtain
\begin{equation}
\label{e:uniformlimit}
\lim_{s\to P_{x_0}^{+}}\Delta_A(s)
=
\lim_{s\to P_{x_0}^{+}}\Delta(s)
=
0.
\end{equation}
Thus, if in addition~$\mu(\partial A) = \mu(\partial f(A)) = 0$ and
the closure of~$A$ is contained in~$I \setminus (\Tur(f)\cup \partial
I)$, then we have
\begin{equation}
\label{e:outcrit}
\mu(f(A))
=
\int_{A}\exp(P_{x_0}-\varphi) \ d \mu.
 \end{equation}
By~\cite[Lemma~$12.1.3$]{PrzUrb10} or~\cite[Lemma~$3.3$]{DenUrb91f},
it follows that the equality above holds for every Borel subset~$A$ of~$I \setminus
(\Tur(f)\cup \partial I)$ on which~$f$ is injective.

\partn{2}
\emph{For every point~$c$ in~$\Tur(f)\setminus \partial I$ we have}
\begin{equation}
\label{e:conformaloncriticalpoints}
 2\mu(\{f(c)\})\ge \mu(\{c\})\exp(P_{x_0}-\varphi(c)),
 \end{equation}
\emph{and for every~$x$ in~$\partial I$ we have}
 \begin{equation}
\label{e:conformalonboundary}
 \mu(\{f(x)\})\ge \mu(\{x\})\exp(P_{x_0}-\varphi(x)).
\end{equation}

First we prove inequality~\eqref{e:conformaloncriticalpoints}. 
Let~$c$ be in~$\Tur(f)\setminus \partial I$, and let $(C_n)_{
  n=1}^{+\infty}$ be a sequence of compact neighborhoods of~$c$ in~$I$
that decreases~$c$, and so that for each~$n$ the map~$f$ is injective on each connected component of $C_n\setminus \{c\}$ and we have~$\mu(\partial
C_n)=0$. 
Put
$$ C_n^{-} \= C_n\cap (- \infty, c]
\text{ and }
C_n^{+} \= C_n\cap  [c, + \infty). $$
From~\eqref{e:uniformlimit} we obtain
$$ \lim_{s\to P_{x_0}^{+} }\left|m_s(f(C_n^{+})) - \int_{C_n^{+}} \exp(P_{x_0}-\varphi) \ d m_s\right|=0, $$
and
$$ \lim_{s\to P_{x_0}^{+}}\left| m_s(f(C_n^{-}))-\int_{C_n^{-}} \exp(P_{x_0}-\varphi) \ d m_s\right|
=
0. $$
By  the construction of~$m_s$, we have $ m_s(\{c\})\to 0$ as $s\to P_{x_0}^{+}$.
On the other hand, note that~$f(C_n^{-})$ and~$f(C_n^{+})$ are compact.
Thus, if we let $(s_j)_{j=1}^{+\infty}$ be a sequence in $(P_{x_0}, +\infty)$ such that $s_j\to P_{x_0}$ and  $m_{s_j}\to \mu$ as $j\to +\infty$, then
 \begin{align*}
2 \mu(f(C_n))
& \ge
\mu(f(C_n^{+}))+\mu(f(C_n^{-}))
\\ & \ge
\limsup_{j\to +\infty} m_{s_j}(f(C_n^{+}))+\limsup_{j\to +\infty} m_{s_j}(f(C_n^{-}))
\\ & =
\limsup_{j\to +\infty }\int_{C_n^{+}} \exp\left(P_{x_0}-\varphi\right)
\ d m_{s_j} +
\limsup_{j\to +\infty}\int_{C_n^{-}} \exp\left(P_{x_0}-\varphi\right) \ d m_{s_j}
\\ & \ge
\liminf_{j\to +\infty} \left( \int_{C_n} \exp\left(P_{x_0}-\varphi\right) \ d
m_{s_j}
+
\int_{\{c\}} \exp\left(P_{x_0}-\varphi\right) \ d m_{s_j} \right)
\\ & \ge
\liminf_{j\to +\infty}\int_{C_n} \exp\left(P_{x_0}-\varphi\right) \ d m_{s_j}
\\ &  \ge
\int_{C_n} \exp\left(P_{x_0}-\varphi\right) \ d \mu
\\ & \ge
\mu(\{c\})\exp(P_{x_0}-\varphi(c)).
\end{align*}
Letting $n\to +\infty$, we obtain~\eqref{e:conformaloncriticalpoints}.

To prove~\eqref{e:conformalonboundary}, let~$x$ be in~$\partial I$,
and let~$(B_n)_{ n=1}^{+\infty}$  be a sequence of compact
neighborhoods of~$x$ in~$I$ that decreases to~$x$, and so that for
each~$n$ the map~$f$ is injective on $B_n$ and we have~$\mu(\partial B_n)=0$. 
The argument above gives us that for every integer $n\ge 1$ we have
$$\mu(f(B_n))\ge \mu(\{x\}) \exp(P_{x_0}-\varphi(x)). $$
Letting $n\to +\infty$, we obtain~\eqref{e:conformalonboundary}.

\partn{3} \emph{$\mu$ has no atoms.}
Assume that there is~$x_*$ in~$I$ such that $\mu(\{x_*\})\neq 0$.
Since by the construction~$\mu$ is supported on~$X$, the point~$x_*$ is in~$X$.
Since~$f$ has no periodic turning point in~$X$, there is an
integer~$n_1\ge 1$ such that the point~$y_*\=f^{n_1}(x_*)$ satisfies for every integer~$n \ge 0$ that~$f^n(y_*)\not\in \Tur(f)$.
Applying~\eqref{e:outcrit}, \eqref{e:conformaloncriticalpoints}, or~\eqref{e:conformalonboundary} repeatedly, we conclude that
 \begin{equation}
\label{e:posi}
 \mu(\{y_*\})=\mu(\{f^{n_1}(x_*)\})\ge \frac{1}{2^{n_1}}\exp(n_1P_{x_0}-S_{n_1}(\varphi)(x_*))\mu(\{x_*\})>0.
 \end{equation}
and that for every integer~$n\ge 1$
\begin{equation}
\label{e:combi}
\mu(\{f^n(y_*)\})\ge \exp(nP_{x_0}-S_n(\varphi)(y_*))\mu(\{y_*\}).
 \end{equation}

On the other hand, by hypothesis there is an integer~$N \ge 1$ such that
\begin{equation}
\label{e:hyper}
N P_{x_0}>\sup_{X} S_N(\varphi).
\end{equation}
Therefore, for each integer $k\ge 1$ we have by~\eqref{e:combi} with $n=kN$,
\begin{align*}
\mu(\{f^{kN}(y_*)\})
& \ge
\exp\left(kNP_{x_0}-S_{kN}(\varphi)(y_*)\right)\mu(\{y_*\})
\\ & \ge
\exp\left(k\left(N P_{x_0}-\sup_{X}S_{N}(\varphi)\right)\right)\mu(\{y_*\}).
\end{align*}
Together with~\eqref{e:posi} and~\eqref{e:hyper}, this implies
$$ \lim_{k\to +\infty}\mu(\{f^{kN}(y_*)\})
=
+ \infty, $$
a contradiction. 
Thus~$\mu$ is atom-free.

\partn{4} \emph{$\mu$ is $\exp(P_{x_0}-\varphi)$-conformal.}
By part~$3$ and~\eqref{e:outcrit}, for each Borel  subset~$A$ of $I$ on which $f$ is injective, we have
 \begin{align*}
\mu(f(A))
&=
\mu(f(A\setminus (\Tur(f)\cup \partial I)))
\\ & =
\int_{A\setminus (\Tur(f)\cup \partial I)}\exp\left(P_{x_0}-\varphi\right) \ d \mu
\\ & =
\int_{A}\exp\left(P_{x_0}-\varphi\right) \ d \mu.
 \end{align*}
Hence,  $\mu$ is $\exp(P_{x_0}-\varphi)$-conformal.

To prove the last statement of the proposition, suppose~$f$ is
topologically exact on~$X$.
Together with~\eqref{e:outcrit}, with~$\exp(P_{x_0} - \varphi) > 0$,
and the fact that~$\mu$ is atom-free, this implies that the support of~$\mu$ is equal to~$X$.
This completes the proof of the proposition.
\end{proof}

\section{Keller spaces and quasi-compactness}
\label{sec:keller}
We start this section by recalling in~\S\ref{ss:keller spaces properties} some properties of the function spaces defined by Keller in~\cite{Kel85}.
In~\S\ref{ss:quasicompactness} we first recall some of the results
of~\cite{Kel85} that we state as Theorem~\ref{t:quasicompactness}, and
then we deduce from them some known consequences
(Corollary~\ref{c:quasicompactness}) that we use in the next section to prove Theorem~\ref{t:spectral gap}.

Throughout this section, fix a compact subset $X$ of $\R$  endowed
with the distance induced by the norm distance on~$\R$, and fix an
atom-free Borel probability measure~$m$ on~$X$.
We consider the equivalence relation on the space of complex valued functions defined on~$X$, defined by agreement on a set of full measure with respect to~$m$.

\subsection{Properties of Keller spaces}
\label{ss:keller spaces properties}
In the following proposition we compare the function spaces introduced by Keller in~\cite{Kel85}, see~\S\ref{ss:keller spaces}, with several other function spaces.

Given~$p\ge 1$ and a function~$h:X \to \C$, put
$$ \Var_p(h)
\=
\sup\left\{\left( \sum_{i=1}^k|h(x_i)-h(x_{i-1})|^p \right)^{\frac{1}{p}}: k \ge 1,  x_0, \ldots, x_k \in X, x_0<\cdots <x_k \right\} $$
and
$$ \|h\|_{\BV_p} \= \Var_p(h)+\|h\|_{\infty}. $$
The function~$h$ is of \emph{bounded~$p$-variation}, if~$\|h\|_{\BV_{p}}< +\infty$.
Let~$\BV_p$ be the space of all bounded $p$-variation functions defined on~$X$.
Then~$\|\cdot\|_{\BV_{p}}$ is a norm on~$\BV_p$, for which~$\BV_p$ is a Banach space.

Given~$\alpha$ in~$(0, 1]$, denote by~$\Hspace^\alpha$ the space of H{\"o}lder continuous functions of exponent~$\alpha$ defined on~$X$ and taking values in~$\C$.
Then~$\| \cdot \|_{\alpha}$, defined in~\S\ref{ss:statements of results}, is a norm on~$\Hspace^{\alpha}$ and~$(\Hspace^{\alpha}, \| \cdot \|_{\alpha})$ is a Banach space.
Note that the definitions immediately imply that for each~$h$ in~$\Hspace^\alpha$, we have
$$ \Var_{1/\alpha}(h)
\le
\left| \sup X - \inf X \right|^\alpha |h|_\alpha, $$
and therefore
\begin{equation}
\label{e:1}
\| h \|_{\BV_{1/\alpha}}
\le
\max \left\{1, \left| \sup X - \inf X \right|^\alpha \right\} \|h\|_\alpha.  
\end{equation}

\begin{prop}
\label{p:keller spaces properties}
Fix~$A > 0$, let~$X$ be a compact subset of~$\R$, and let~$m$ be an atom-free Borel probability measure on~$X$.
Then for each~$\alpha$ in~$(0,1]$, the space~$\Hspace^{\alpha, 1}(m)$ defined in~\S\ref{ss:keller spaces} satisfies the following properties:
\begin{enumerate}
\item[1.]
$(\Hspace^{\alpha,1}(m), \|\cdot\|_{\alpha,1})$ is a Banach space;
\item[2.]
For each function $h:X\to \C$ in~$\BV_{1/\alpha}$, we have $\|h\|_{\alpha,1}\le 2^{\alpha}\|h\|_{\BV_{1/\alpha}}$;
\item[3.]
Each function in~$\Hspace^{\alpha,1}(m)$ is essentially bounded.
In fact, there is a constant $C_*>0$ such that each element~$h$
of~$\Hspace^{\alpha,1}(m)$ satisfies~$\| h \|_{\infty} \le C_* \|h\|_{\alpha,1}$.
\end{enumerate}
\end{prop}

Note that by combining~\eqref{e:1} with parts~$2$ and~$3$ of the proposition above, we obtain
$$ \Hspace^\alpha
\subset
\BV_{1/\alpha}
\subset
\Hspace^{\alpha,1}(m)
\subset
\Lspace^{\infty}(m). $$
We conclude that each of the spaces~$\BV_{1/\alpha}$ and~$\Hspace^{\alpha,1}(m)$ is dense in~$\Lspace^1(m)$.

Part~$1$ of Proposition~\ref{p:keller spaces properties} is part~$b$ of~\cite[Theorem~$1.13$]{Kel85}. 
Since this property is important for what follows, we provide a proof.

The rest of this subsection is devoted to the proof of Proposition~\ref{p:keller spaces properties}.

\begin{lemm}
\label{l:oscbyvar}
For all $p\ge 1$, $h:X\to \C$, and $\eps>0$, we have
$$ \int_X \osc(h,\eps,x)^p \ dm(x)
\le
2 \eps \Var_p(h)^p. $$
\end{lemm}
\begin{proof}
Put $a\=\inf X$ and for each $t$ in $[0, m(X)]$, put
$$x(t) \= \sup\{y\in X: d(y,a)=t\}. $$
Suppose first $\eps\ge m(X)/2$.
Using that for every $x$ in $X$ we have  $\osc(h,\eps, x)^p\le \Var_p(h)^p$, we obtain 
$$ \int_X \osc(h,\eps,x)^p \ dm(x)
\le
m(X) \Var_p(h)^p\le 2\eps \Var_p(h)^p. $$
It remains to consider the case $\eps< m(X)/2$. 
For every~$\xi$ in~$[0,2\eps]$, put
$$ n_{\eps}(\xi)
\=
\max \{\text{ nonnegative integer } n : \xi+2n \eps\le m(X)\}. $$
Since the balls $\left( B_d(x(\xi+2k\eps), \eps) \right)_{k=0}^{n_\eps(\xi)}$ are pairwise disjoint, we have
$$ \sum_{k=0}^{n_{\eps}(\xi)}\osc\left(h,\eps, x(\xi+2k\eps)\right)^{p}
\le
\Var_p(h)^p. $$
It follows that
\begin{align*}
\int_X\osc(h,\eps, x)^p \ d m(x)
=
\int_{0}^{2\eps}\sum_{k=0}^{n_{\eps}(\xi)}\osc(h,\eps, x(\xi+2k\eps'))^p \ d \xi
\le
2 \eps \Var_p(h)^p,
\end{align*}
and so we obtain the lemma.
\end{proof}

\begin{lemm}[\cite{Kel85}, Lemma~$1.12$]
\label{limit}
Let~$\alpha$ be in~$(0,1]$, and let~$(h_n)_{n=1}^{+\infty}$ be a sequence in~$\Hspace^{\alpha,1}(m)$. 
If there is~$h$ in~$\Lspace^1(m)$ such that~$(h_n)_{n=1}^{+\infty}$ converges to~$h$ in~$\Lspace^1(m)$, then
$$ |h|_{\alpha,1}
\le
\liminf_{n\to +\infty}|h_n|_{\alpha,1}. $$
\end{lemm}

\begin{proof}[Proof of Proposition~\ref{p:keller spaces properties}]
To prove part~$1$ it suffices to check that $\Hspace^{\alpha, 1}(m)$ is complete with respect to~$\|\cdot\|_{\alpha,1}$. 
Let~$(h_n)_{n=1}^{+\infty}$ be a Cauchy sequence in~$\Hspace^{\alpha,1}(m)$. 
Then~$(h_n)_{n=1}^{+\infty}$ is also a Cauchy sequence
in~$\Lspace^1(m)$, and therefore there is~$h$ in~$\Lspace^1(m)$  such that $\|h_{n}-h\|_1\to 0$ as $n\to +\infty$.
By Lemma~\ref{limit}, we have
$$ |h|_{\alpha,1}
\le
\liminf_{n\to +\infty}|h_n|_{\alpha,1}
<
+\infty. $$
It follows that~$h$ is in~$\Hspace^{\alpha,1}(m)$. 
To complete the proof of part~$1$, it is enough to prove that for every $\delta>0$ there is $N>0$ such that for each integer $n\ge N$ we have
$$ |h_n - h|_{\alpha,1} \le \delta. $$
In fact, since $(h_n)_{n=1}^{+\infty}$ is a Cauchy sequence in $\Hspace^{\alpha,1}(m)$,  there is  $N>0$ such that for each pair of integers $k, n\ge N$, we have $\|h_k-h_n\|_{\alpha,1}<\delta$.
Fix $n\ge N$, and note that $h-h_k+(h_n-h)$ converges to $h_n-h$ in $\Lspace^1(m)$ as $k\to +\infty$.
It follows from Lemma~\ref{limit} again that
$$ |h_n-h|_{\alpha,1}
\le
\liminf_{k\to +\infty}|h-h_k+(h_n-h))|_{\alpha,1}
=
\liminf_{k\to +\infty}|h_n-h_k|_{\alpha,1}
\le
\delta. $$
The proof of part~$1$ is complete.

Let us prove part~$2$. 
By  H{\"o}lder's integral inequality and Lemma~\ref{l:oscbyvar} with $p=1/\alpha$,  for every $\eps$ in $(0,A]$ we have
\begin{align*}
\osc_{1}(h,\eps)
\le
\left( \int_X \osc(h,\eps,x)^{1/\alpha} \ dm(x) \right)^{\alpha}
\le
2^{\alpha}\eps^{\alpha} \Var_{1/\alpha}(h).
\end{align*}
It follows that $|h|_{\alpha,1}\le 2^{\alpha} \Var_{1/\alpha}(h)$.
On the other hand, since~$\|h\|_1\le \|h\|_{\infty}$, we have
$$ \|h\|_{\alpha,1}
=
\|h\|_{1} + |h|_{\alpha,1}
\le
\|h\|_{\infty} + 2^{\alpha} \Var_{1/\alpha}(h)\le 2^{\alpha}\|h\|_{\BV_{1/\alpha}}. $$

It remains to prove part~$3$.
Let~$h$ be in~$\Hspace^{\alpha,1}(m)$ and fix~$\varepsilon > 0$.
Then there are subsets~$X_1$ and~$X_2$ of~$X$ of positive measure for~$m$, such that
$$ \sup_{X_1} |h| \le \| h \|_1
\text{ and }
\inf_{X_2} |h| \ge \esssup_{X} |h| - \varepsilon. $$
It follows that
\begin{equation}
 \label{boun}
 \begin{split}
\esssup_{X} |h| - \varepsilon
& \le
\inf_{X_2} |h|
\\ & \le
\sup_{X_1} |h|
+ 
\esssup_{(x_1, x_2) \in X_1 \times X_2} |h(x_1) - h(x_2)|
\\ & \le
\| h \|_1
+
\esssup_{(x_1, x_2) \in X_1 \times X_2} |h(x_1) - h(x_2)|.
 \end{split} 
\end{equation}
Let~$N$ be the least integer such that~$N \ge m(X) / (2 A)$, put~$\ell \= m(X) / (4N)$, and note that~$A \ge 2\ell$.
Given~$\xi$ in~$[0,2\ell]$, let~$n(\xi)$ be the largest integer~$n \ge 0$ such that~$\xi + 2n \ell \le m(X)$.
Note that if for each $t \ge 0$ we put $x(t)\=\sup\{y\in X: d(y, \inf X)\le t\}$, then $\bigcup_{k=0}^{n(\xi)}B_d(x(\xi+2k\ell),A)$ has full measure in~$X$ with respect to~$m$.
We thus have
$$ \esssup_{(x_1, x_2) \in X_1 \times X_2} |h(x_1) - h(x_2)|
\le
\sum_{k=0}^{n(\xi)}\osc(h, A, x(\xi+2k\ell)). $$
Since this holds for every~$\xi$ in~$[0, 2\ell]$, we have
\begin{align*}
\esssup_{(x_1, x_2) \in X_1 \times X_2} |h(x_1) - h(x_2)|
& \le
\frac{1}{2 \ell}\int_{0}^{2 \ell} \sum_{k=0}^{n(\xi)}\osc(h, A, x(\xi+2k\ell)) \ d \xi
\\ & =
\frac{1}{2\ell}\int_{ X} \osc(h, A, x) \ d m(x)
\\ & \le
\frac{1}{2 \ell}|h|_{\alpha,1} A^\alpha.
\end{align*}
Together with~\eqref{boun} this implies,
$$ \esssup_{X} |h| - \varepsilon
\le
\|h\|_1 + \frac{A^\alpha}{2\ell} |h|_{\alpha,1}
\le
\max \left\{ 1, \frac{A^\alpha}{2\ell} \right\} \|h\|_{\alpha,1}. $$
Since this holds for every~$\varepsilon >0$, this proves part~$3$ with $C_* = \max \left\{ 1, \frac{A^\alpha}{2\ell} \right\}$, and completes the proof of the proposition.
\end{proof}

\subsection{Quasi-compactness and spectral gap}
\label{ss:quasicompactness}
In this subsection, let~$I$ be a compact interval of~$\R$, $N\ge 2$ an integer,  and $\sP'\=\{I'_1, \cdots, I'_N\}$ a partition of $I$ into intervals. 
Let $T:I \to I$ be a transformation on~$I$ that is continuous and monotone on each~$I'_i$ in~$\sP'$.
Furthermore, let~$X$ be a compact subset of~$I$ such that~$T^{-1}(X) = X$, for each~$i$ in $\{1,\cdots, N\}$ put $I_i:=I'_i\cap X$, and put $\sP:=\{I_1, \cdots, I_N\}$.
Fix $p\ge 1$, let $g: X\to [0, +\infty)$ be a function of bounded $p$-variation, and let $\sL_g$  be the operator acting on the space
$$ \text{Eb}(X)
\=
\{h:X \to \C \text{ measurable and bounded in absolute value} \}, $$
defined by
\begin{equation}
\label{e:otor}
\sL_g(h)(x)
\=
\sum_{y\in T^{-1}(x)}h(y)g(y)
=
\sum_{i \in \{1, \ldots, N \}, x \in T(I_i)} (h\cdot g) \circ T_{|_{I_i}}^{-1}(x).
 \end{equation}
Assume in addition that there is an atom-free Borel probability measure~$m$ on~$X$ such that the following properties hold:
\begin{enumerate}
\item[H$1$.]
For each $I_i$ in $\sP$, the map~$T|_{I_i}^{-1}$ is non-singular with respect to~$m$, so that for every subset~$E$ of~$I_i$ of measure zero, the set~$\left(T|_{I_i}^{-1}\right)^{-1}(E) = T(E)$ is also of measure zero;
\item[H$2$.]
On a set of full measure with respect to~$m$, we have
$$ g^{-1}
=
\sum_{i=1}^N \frac{d \left( T_{|_{I_i}}^{-1} \right)_*m }{d m} ; $$
\item[H$3$.]
For each~$h$ in~$\text{Eb}(X)$ we have $\int_X \sL_g(h) \ d m = \int_X h \ d m$, and $\sL_g$ extends to a positive linear map from $\Lspace^1(m)$ to itself satisfying $\|\sL_g(h)\|_1\le \|h\|_1$.
\end{enumerate}

In the following theorem, we gather several results from~\cite{Kel85}.
\begin{theo}[\cite{Kel85}, Theorems~$3.2$ and~$3.3$]
\label{t:quasicompactness}
Let~$T$, $g$, $\sL_g$, and~$m$ be as above, and assume that there is an integer $n\ge 1$ such that the function
$$ g_n(x)
\=
g(x) \cdot \cdots \cdot g(T^{n-1}(x)) $$
satisfies $\sup_{X}g_n < 1$.
Then there is an integer~$k \ge 1$ and constants $A>0$, $\beta$ in~$(0,1)$, and~$C>0$, such that for every function~$h$ in~$\Hspace^{1/p, 1}(m)$, we have
\begin{equation}
\label{twon}
\|\sL^k_g(h)\|_{1/p,1}\le \beta \|h\|_{1/p,1}+ C\|h\|_1.
\end{equation}
Moreover, the following properties hold:
\begin{enumerate}
\item[1.]
The set~$\sE$ of eigenvalues of~$\sL_{g}|_{\Lspace^1(m)}$ of modulus~$1$ is finite.
Moreover, for each~$\lambda$ in~$\sE$, the space
$$ \Espace(\lambda)
\=
\{h \in \Lspace^1(m) : \sL_{g}(h)=\lambda h\} $$
is contained in~$\Hspace^{1/p, 1}(m)$ and it is of finite dimension;
\item[2.]
If for each~$\lambda$ in~$\sE$ we denote by~$\sP(\lambda)$ the projection in~$\Lspace^1(m)$ to~$\Espace(\lambda)$, then the operator
$$ \sQ \= \sL_{g} - \sum_{\lambda \in \sE} \sP(\lambda) $$
satisfies~$\sup \{ \| \sQ^n \|_1 : n \ge 0 \text{ integer} \} < + \infty$.
Moreover, $\sQ$ maps~$\Hspace^{1/p, 1}(m)$ to itself, and there is~$\rho$ in~$(0, 1)$ and a constant~$M > 0$ such that for every integer~$n \ge 0$ we have~$\| \sQ^n \|_{\alpha, 1} \le M \rho^n$.
Finally, for each~$\lambda$ in~$\sE$ the operators~$\sQ \sP(\lambda)$ and~$\sP(\lambda) \sQ$ are both identically zero, and for each~$\lambda'$ in~$\sE$ different from~$\lambda$ the operators~$\sP(\lambda) \sP(\lambda')$ and~$\sP(\lambda') \sP(\lambda)$ are also identically zero.
\item[3.]
The set~$\sE$ contains~$1$, and if we put $h \= \sP(1)(\1)$, then~$\nu \= hm$ is a probability measure that is invariant by~$T$ and that is an equilibrium state of~$T_{|_X}$ for the potential~$\log g$.
\end{enumerate}
\end{theo}

The following corollary follows from the previous theorem using known arguments. 
We include its proof for completeness.

\begin{coro}
\label{c:quasicompactness}
Under the assumptions of Theorem~\ref{t:quasicompactness}, and assuming  in addition that $T$ is topologically exact on $X$, we have the following properties:
\begin{enumerate}
\item[1.]
The number~$1$ is an eigenvalue of~$\sL_g$ of algebraic multiplicity~$1$.
Moreover, there is~$\rho$ in~$(0, 1)$ such that the spectrum of~$\sL_g|_{\Hspace^{1/p, 1}(m)}$ is contained in~$B(0, \rho) \cup \{ 1 \}$.
\item[2.]
There is a constant~$C > 0$ such that for every bounded measurable function~$\phi : X \to \C$, and every function~$\psi$ in~$\Hspace^{1/p, 1}(m)$, the measure~$\nu$ given by part~$3$ of Theorem~\ref{t:quasicompactness} satisfies for every integer~$n \ge 1$ that
$$ C_n(\phi, \psi) \le C \| \phi \|_{\infty} \| \psi \|_{1/p, 1}
\rho^n. $$
\item[3.]
Given~$\psi$ in~$\Hspace^{1/p, 1}(m)$, for each~$\tau$ in~$\C$ the operator~$\sL_\tau$ defined by
$$ \sL_{\tau} (h)
\=
\sL_g \left( \exp(\tau \psi) \cdot h \right) $$
maps~$\Hspace^{1/p, 1}(m)$ to itself and the
restriction~$\sL_{\tau}|_{\Hspace^{1/p, 1}(m)}$ is bounded.
Moreover, $\tau \mapsto \sL_{\tau}|_{\Hspace^{1/p, 1}(m)}$ is
analytic in the sense of Kato on~$\C$, and the spectral radius
of~$\sL_{\tau}|_{\Hspace^{1/p, 1}(m)}$ depends on a real
analytic way on~$\tau$ on a neighborhood of~$\tau = 0$.
\end{enumerate}
\end{coro}

The proof of this corollary is after the following lemma.

\begin{lemm}
  \label{l:product in keller}
Let~$\alpha$, $m$, and~$\Hspace^{\alpha, 1}(m)$ be as above, and let~$C_*$ be given by Proposition~\ref{p:keller spaces properties}.
Then for every~$h$ and~$g$ in~$\Hspace^{\alpha, 1}(m)$, we have
$$ \| h \cdot g \|_{\alpha, 1}
\le
2 C_* \| h \|_{\alpha, 1} \cdot \| g \|_{\alpha, 1}. $$
\end{lemm}
\begin{proof}
Using that each of the functions~$h$ and~$g$ is represented by a bounded function, we have
$$ \osc(h \cdot g, \varepsilon, x)
\le
\| h \|_{\infty} \osc(g, \varepsilon, x) + \| g \|_{\infty} \osc(h, \varepsilon, x). $$
We thus have
$$ | h \cdot g |_{\alpha, 1}
\le
\| h \|_{\infty} \cdot | g |_{\alpha, 1} + \| g \|_{\infty} \cdot | h
|_{\alpha, 1}, $$
and using part~$3$ of Proposition~\ref{p:keller spaces properties} twice, we have
\begin{align*}
  \| h \cdot g \|_{\alpha, 1}
& \le
\| h \|_{\infty} \cdot \| g \|_1 + \| h \|_{\infty} \cdot | g |_{\alpha, 1} + \| g \|_{\infty} \cdot | h |_{\alpha, 1}
\\ & \le
\| h \|_{\infty} \cdot \| g \|_{\alpha, 1} + C_* | h |_{\alpha, 1} \| g \|_{\alpha, 1}
\\ & \le
2 C_* \| h \|_{\alpha, 1} \cdot \| g \|_{\alpha, 1}.
\end{align*}
\end{proof}

\begin{proof}[Proof of Corollary~\ref{c:quasicompactness}]
By part~$3$ of Theorem~\ref{t:quasicompactness} the function~$h \= \sP(1)(\pmb{1})$
satisfies~$\int_{X} h \ d m = 1$ and the measure~$\nu \= h m$ is a probability.
On the other hand, each of the spaces
$$ \Espace_0
\=
\left\{\psi \in \Hspace^{1/p, 1}(m) : \int_{X} \psi \ d m=0 \right\}
\text{ and }
\Espace_1
\=
\left\{ \alpha h : \alpha \in \C \right\} $$
is invariant by~$\sL_g$.
Moreover, since each function~$\psi$ in $\Hspace^{1/p, 1}(m)$ can be decomposed as
$$ \psi
=
\left(\int_{X} \psi \ d m\right)h + \left(\psi-\left( \int_{X} \psi \ d m\right)h\right), $$
we have~$\Hspace^{1/p, 1}(m)=E_1 \oplus E_0$.

\partn{1}
Let~$\lambda$ be an eigenvalue of~$\sL_g|_{\Hspace^{1/p,1}(m)}$
satisfying~$|\lambda|=1$, and let~$\phi$ be a nonzero element of~$\Hspace^{1/p, 1}(m)$ such that
$\sL_g(\phi)=\lambda \phi$.
By hypothesis~H$3$, for each integer $n\ge 1$ we have
\begin{align*}
\int_{X}|\phi| \ d m
=
\int_{X}|\lambda^n \phi| \ d m
=
\int_{X}|\sL_g^n(\phi)| \ d m
\le
\int_{X}\sL_g^n\left(|\phi|\right) \ d m
=
\int_{X}|\phi| \ d m,
\end{align*}
and therefore
$$ \int_{X} \sL_g^n(|\phi|)-|\sL_g^n(\phi)| \ d m =0. $$
Since $\sL_g^n(|\phi|)-|\sL_g^n(\phi)|\ge 0$, it follows that we have
\begin{equation}
\label{e:argument invariance}
\sL_g^n(|\phi|) = |\sL_g^n(\phi)| = |\phi|
\end{equation}
on a set of full measure with respect to~$m$.

In part~$1.1$ below we prove that~$\phi$ is nonzero on a set of
full measure with respect to~$m$, and in part~$1.2$ we show that there
is~$\theta_0$ in~$\R$ such that~$\phi = \exp(i\theta_0)
| \phi|$ in~$\Hspace^{\alpha, 1}(m)$.
Using these facts, we complete the proof of part~$1$ of the corollary in part~$1.3$.

\partn{1.1}
Since~$\phi$ is nonzero, there is~$\kappa_0 > 0$ such that~$\{ x \in X :
|\phi| \ge \kappa_0 \}$ has positive measure with respect to~$m$.
Let~$Y$ be the set of density points of this set, so~$m(Y) > 0$, and put
$$ \varepsilon_0
\=
\min \left\{ \left( \frac{ m(Y) \kappa_0}{2 \| | \phi | \|_{\alpha, 1}}
\right)^{\frac{1}{\alpha}}, A \right\}. $$
Note that for each~$y$ in~$Y$, the number
$$ \kappa(y) \= \essinf \{ |\phi(x)| : x \in B_d(y, \varepsilon_0) \} $$
satisfies~$\kappa(y) \le \kappa_0$.
On the other hand, we have
$$ \osc(|\phi|, \varepsilon_0, y)
\ge
\kappa_0 - \kappa(y), $$
so
$$ \int_Y \kappa_0 - \kappa(y) \ dm(y)
\le
\osc_1(|\phi|, \varepsilon_0)
\le
\varepsilon_0^{\alpha} \| |\phi| \|_{\alpha, 1}
\le
m(Y) \kappa_0 /2. $$
This implies that there is~$y_0$ in~$Y$ such that
\begin{equation}
  \label{e:bound on interval}
\essinf \{ |\phi(x)| : x \in B_d(y_0, \varepsilon_0) \}
= 
\kappa(y_0)
\ge
\kappa_0 /2.  
\end{equation}
Since by hypothesis~$f$ is topologically exact on~$X$, there is an integer~$n \ge 1$ such that~$T^n(B_d(y_0, \varepsilon_0)) = X$.
Combined with hypothesis~H$2$ and~\eqref{e:argument invariance}, the
estimate~\eqref{e:bound on interval} implies that~$| \phi |$ is nonzero on a set of full measure with
respect to~$m$.

\partn{1.2}
By part~$1.1$ there is~$\upsilon > 0$ such that
$$ W \= \{ x \in X : |\phi(x)| \ge \upsilon \} $$
satisfies~$m(W) \ge 1/2$.
Let~$W'$ be the set of density points of~$W$, so~$m(W') \ge 1/2$.

For each~$x$ in~$X$ let~$\theta(x)$ in~$\R$ be such that~$\phi(x) =
\exp(i \theta(x))|\phi(x)|$.
Suppose by contradiction that the function~$\theta : X \to \R$ so
defined is not constant on a set of full measure with respect to~$m$.
Then there are disjoint closed intervals~$\Theta$ and~$\Theta'$ such
that the sets~$\theta^{-1}(\Theta)$ and~$\theta^{-1}(\Theta')$ are disjoint, and
such that each of these sets has positive measure with respect to~$m$.
Combined with hypothesis~H$2$, property~\eqref{e:argument invariance} with~$n
= 1$ implies that there is a subset~$Z$ (resp.~$Z'$) of~$\theta^{-1}(\Theta)$
(resp.~$\theta^{-1}(\Theta')$) of full measure such that
$$ T^{-1}(Z) \subset Z
\quad
\left( \text{resp. } T^{-1}(Z') \subset Z' \right). $$
It follows that each of the sets~$Z$ and~$Z'$ is dense in~$X$.
Thus, if we denote by~$\delta$ the distance between~$\Theta$
and~$\Theta'$ in~$\R / 2\pi \Z$, then for every~$y$ in~$W'$ and
every~$\varepsilon$ in~$(0, A]$, we have
$$ \osc(\phi, \varepsilon, y)
\ge
2 \upsilon \sin(\delta/2). $$
Therefore
$$ \| \phi \|_{\alpha, 1}
\ge
\frac{\osc_1(\phi, \varepsilon)}{\varepsilon^{\alpha}}
\ge
\frac{m(W') (2 \upsilon \sin(\delta/2))}{\varepsilon^{\alpha}}
\ge
\frac{\upsilon \sin(\delta/2)}{\varepsilon^{\alpha}}. $$
Since this holds for an arbitrary~$\varepsilon$ in~$(0, A]$, we obtain a
contradiction.
This contradiction shows that the function~$\theta$ is constant on a set of full
measure with respect to~$m$.

\partn{1.3}
By part~$1.2$ there is~$\theta_0$ in~$\R$ such that~$\phi = \exp(i
\theta_0) | \phi|$ in~$\Hspace^{\alpha, 1}(m)$.
It follows that~$\sL_g \left( |\phi| \right) = \lambda |\phi|$ is nonnegative, and
therefore that~$\lambda = 1$.
Since by part~$3$ of Theorem~\ref{t:quasicompactness} the number~$1$
is an eigenvalue of~$\sL_g|_{\Hspace^{1/p, 1}(m)}$, this proves that the number~$1$ is the only eigenvalue of~$\sL_{g}$ of modulus~$1$.

The existence of~$\rho$ in~$(0, 1)$ such that the spectrum of~$\sL_g|_{\Hspace^{1/p, 1}(m)}$ is contained in~$B(0, \rho) \cup \{ 1 \}$ follows from part~$2$ of Theorem~\ref{t:quasicompactness}.

It remains to prove that the algebraic multiplicity of~$1$ as an eigenvalue of~$\sL_g|_{\Hspace^{1/p, 1}(m)}$ is $1$.
Denote by~$\Id$ the identity operator of~$\Hspace^{1/p, 1}(m)$, and let~$\phi$ be in the kernel of~$\left( \sL_g|_{\Hspace^{1/p, 1}(m)} - \Id \right)^2$.
Then~$\tphi\=\sL_g(\phi)-\phi$ satisfies $\sL_g(\tphi) = \tphi$.
Suppose~$\tphi$ is nonzero.
Then we can apply parts~$1.1$ and~$1.2$ with~$\phi$ replaced
by~$\tphi$, to conclude that there is~$\ttheta_0$ in~$\R$ such
that~$\tphi = \exp \left( i \ttheta_0 \right) |\tphi|$ in~$\Hspace^{\alpha, 1}(m)$.
Using hypothesis~H$3$ we obtain
\begin{equation*}
  \begin{split}
0
& < 
\int_{X} |\tphi| \ d m
\\ & =
\exp \left( -i \ttheta_0 \right) \int_{X}\sL_g(\phi)-\phi \ d m
\\ & =
\exp \left( -i \ttheta_0 \right) \left( \int_{X}\sL_g(\phi) \ d m-
  \int_{X}\phi \ d m \right)
\\ & =
0.
  \end{split}
\end{equation*}
This contradiction proves that $\sL_g(\phi)-\phi = \tphi$ is zero, and completes the proof of part $1$.

\partn{2}
Let~$C_*$ be the constant given by part~$3$ of Proposition~\ref{p:keller spaces properties} and let~$\rho$ and~$M$ be the constants given by part~$2$ of Theorem~\ref{t:quasicompactness}.
Putting~$\hpsi=\psi-\int_X \psi \ d \nu$, we have
\begin{equation*}
  \begin{split}
C_n(\phi, \psi)
& =
\left|\int_{X}\phi \circ f^n \cdot \hpsi \ d\nu \right|
\\ & =
\left|\int_{X}\left(\phi \circ f^n\right) \cdot \hpsi \cdot h \ dm \right|
\\ & =
\left|\int_{X}\sL_g^n\left(\left(\phi \circ f^n\right) \cdot \hpsi \cdot h\right) \ d m \right|
\\ & =
\left|\int_{X}\phi \cdot \sL_g^n\left(\hpsi \cdot h\right) \ d m \right|
\\ & \le
\|\phi\|_{\infty} \cdot \left\| \sL_g^n\left(\hpsi \cdot h  \right) \right\|_1
\\ & \le
\|\phi\|_{\infty} \cdot \left\|\sL_g^n\left(\hpsi \cdot h  \right)\right\|_{1/p,1}.    
  \end{split}
\end{equation*}
Noting that~$\hpsi \cdot h$ is in~$E_0$ and using part~$2$ of Theorem~\ref{t:quasicompactness}, we conclude that
\begin{equation}
\label{m:1}
C_n(\phi, \psi)
\le
M \|\phi\|_{\infty}\|\hpsi \cdot h\|_{1/p, 1} \rho^n.
\end{equation}
On the other hand, by Lemma~\ref{l:product in keller} we have
\begin{equation*}
  \begin{split}
\|\hpsi \cdot h  \|_{1/p, 1}
& \le
2 C_* \|\hpsi \|_{1/p, 1} \cdot \| h \|_{1/p, 1}
\\ & \le
\left( 2 C_* \| h \|_{1/p, 1} \right) \left( \| \psi \|_{1/p, 1} + \| \psi \|_1 \cdot \| h \|_{\infty} \right)
\\ & \le
\left( 2 C_* \| h \|_{1/p, 1} \left(1 + \| h \|_{\infty} \right) \right) \| \psi \|_{1/p,1}.    
  \end{split}
\end{equation*}
Together with~\eqref{m:1} this implies the desired inequality with $C
= 2 M C_*\|h\|_{1/p,1} \left( 1 + \| h \|_{\infty} \right)$.

\partn{3}
Let~$C_*$ be the constant given by part~$3$ of Proposition~\ref{p:keller spaces properties}.
Observe that for each~$\tau$ in~$\C$, we have
$$ | \exp(\tau \psi) |_{1/p, 1}
\le
\exp(|\tau| \cdot \| \psi \|_{\infty}) |\tau| \cdot | \psi |_{1/p, 1},
$$
so the function~$\exp(\tau \psi)$ is in~$\Hspace^{1/p, 1}(m)$.
Thus, by Lemma~\ref{l:product in keller} for every~$\chi$ in~$\Hspace^{1/p, 1}(m)$ we have
$$ \left\| \sL_{\tau} (\chi) \right\|_{1/p, 1}
\le
\left( 2 C_* \left\| \sL_{g} \right\|_{1/p, 1} \cdot \|\exp(\tau \psi) \|_{1/p, 1} \right) \| \chi \|_{1/p, 1}. $$
This proves that~$\sL_{\tau}$ maps~$\Hspace^{1/p, 1}(m)$ to itself and that~$\sL_{\tau}|_{\Hspace^{1/p, 1}(m)}$ is bounded.

To prove that~$\tau \mapsto \sL_{\tau}|_{\Hspace^{1/p, 1}(m)}$ is analytic in the sense of Kato, for each~$\varepsilon$ in~$\C$ let~$\eta_{\varepsilon} : \C \to \C$ be defined by~$\eta_{\varepsilon}(z) \= \frac{\exp(\varepsilon z) - 1}{\varepsilon} - z$ and put~$\psi_{\varepsilon} \= \eta_{\varepsilon} \circ \psi$.
Noting that~$D \eta_{\varepsilon}(z) = \exp(\varepsilon z) - 1$, we have
$$ |\psi_\varepsilon|
\le
\left( \exp ( |\varepsilon| \cdot \| \psi \|_{\infty}) - 1 \right) |\psi| $$
on~$X$, and
$$ | \psi_{\varepsilon}|_{1/p, 1}
\le
\left( \exp( |\varepsilon| \cdot \| \psi \|_{\infty}) - 1 \right) |
\psi |_{1/p, 1}. $$
It follows that
\begin{equation}
\label{e:perturbed potential}
\| \psi_{\varepsilon} \|_{1/p, 1}
\le
\left( \exp( |\varepsilon| \cdot \| \psi \|_{\infty}) - 1 \right) \| \psi \|_{1/p, 1}.
\end{equation}
On the other hand, if for each~$\tau$ in~$\C$ we define the operator~$\sD_\tau$ by~$\sD_\tau(\chi) \= \sL_{\tau}(\psi \cdot \chi)$, then for every~$\varepsilon$ in~$\C$ and every~$\chi$ in~$\Hspace^{1/p, 1}(m)$ we have
$$ \frac{\sL_{\tau + \varepsilon}(\chi) - \sL_{\tau}(\chi)}{\varepsilon} - \sD_{\tau}(\chi)
=
\sL_\tau (\psi_\varepsilon \cdot \chi). $$
Combined with~\eqref{e:perturbed potential}, we have by Lemma~\ref{l:product in keller}
\begin{align*}
  \left\| \frac{\sL_{\tau + \varepsilon}(\chi) - \sL_{\tau}(\chi)}{\varepsilon} - \sD_{\tau}(\chi) \right\|_{1/p, 1}
& \le
2 C_* \left\| \sL_{\tau} \right\|_{1/p, 1} \| \psi_{\varepsilon} \|_{1/p, 1} \| \chi \|_{1/p, 1}
\\ & \le
\left( 2 C_* \left\| \sL_{\tau} \right\|_{1/p, 1} \| \psi \|_{1/p, 1}
\right)
\\ & \quad \cdot
\left( \exp( |\varepsilon| \cdot \| \psi \|_{\infty} ) - 1 \right)  \| \chi \|_{1/p, 1}.
\end{align*}
This implies that the operator norm~$\left\| \frac{\sL_{\tau +
      \varepsilon} - \sL_{\tau}}{\varepsilon} - \sD_{\tau}
\right\|_{1/p, 1}$ converges to~$0$ as~$\varepsilon$ converges to~$0$,
and completes the proof that~$\tau \mapsto \sL_{\tau}|_{\Hspace^{1/p,
    1}(m)}$ is analytic in the sense of Kato.

That the spectral radius of~$\sL_{\tau}|_{\Hspace^{1/p,1}(m)}$ depends
on a real analytic way on~$\tau$ on a neighborhood of~$\tau = 0$ follows from part~$1$
and from the fact that~$\tau \mapsto \sL_{\tau}|_{\Hspace^{1/p, 1}(m)}$ is analytic in the sense of Kato, see for example~\cite[Theorem~XII.$8$]{ReeSim78}.
This completes the proof of the corollary.
\end{proof}

\section{Proof of Theorem~\ref{t:spectral gap}}
\label{s:proof of main theorem}
The purpose of this section is to prove Theorem~\ref{t:spectral gap}.
Throughout this section, fix an interval map~$f : I \to I$ in~$\sA$ that is topologically exact on its Julia set~$J(f)$. 
Then there is $N\ge 2$, and  a partition~$\sP'\=\{I'_1, \cdots, I'_N\}$ of~$I$ into intervals, such that~$f$ is continuous and strictly monotone on each~$I'_i$ in~$\sP'$.
For each $i$ in $\{1,\cdots,N\}$, put $I_i\= I'_i \cap J(f)$.

Let~$\alpha$ be in~$(0, 1]$, let~$\varphi:J(f)\to \R$ be a H{\"o}lder continuous potential of exponent~$\alpha$ that is hyperbolic for $f$, and let~$N \ge 1$ be an integer such that the function~$\tvarphi \= \frac{1}{N} S_N(\varphi)$ satisfies~$\sup_{J(f)} \tvarphi < P(f, \varphi)$.
By part~$1$ of Lemma~\ref{l:onetime}, the function~$\tvarphi$ is H{\"o}lder continuous of exponent~$\alpha$, the potentials $\varphi$ and~$\tvarphi$ share the same equilibrium states, and $P(f, \tvarphi) = P(f, \varphi)>\sup_{J(f)}\tvarphi$.
On the other hand, let~$\chi : J(f) \to \R$ be a H{\"o}lder continuous function, and note that~$\tchi \= \frac{1}{N} S_N(\chi)$ is H{\"o}lder continuous of the same exponent as~$\chi$ and that for every~$t$ in~$\R$ we have~$P(f, \tvarphi + t\tchi) = P(f, \varphi + t \chi)$.
Thus, replacing~$\varphi$ by~$\tvarphi$ and~$\chi$ by~$\tchi$ if necessary, throughout the rest of this section we can assume~$\sup_{J(f)} \varphi < P(f, \varphi)$.

Recall that the operator~$\cL_{\varphi}$ is defined by
$$ \cL_{\varphi}(\psi)(x)
\=
\sum_{y\in f^{-1}(x)} \exp\left(\varphi(y)\right) \psi(y), $$
see~\S\ref{ss:statements of results}, and that~$\hcL_{\varphi} \= \exp(-P(f,\varphi))\cL_\varphi$.
Note that if we put $g\=\exp(\varphi-P(f,\varphi))$, then~$\hcL_\varphi$ coincides with the operator~$\sL_g$ defined in~\eqref{e:otor} with~$T$ replaced by~$f$, and~$X$ replaced by $J(f)$.

The following lemma is well-known.
We include its short proof for completeness.
\begin{lemm}
\label{l:operator}
Let~$\mu$ be an atom-free $\exp(P(f,\varphi)-\varphi)$-conformal measure for~$f$.
Then for every function~$\psi$ in~$\Lspace^1(\mu)$, we have
$$ \int_{J(f)} \hcL_\varphi(\psi) \ d \mu
=
\int_{J(f)} \psi \ d\mu. $$
\end{lemm}
\begin{proof}
Using that~$\mu$ is $\exp(P(f,\varphi)-\varphi)$-conformal, for every~$i$ in $\{1,2,\cdots,N\}$ and every measurable subset~$S$ of~$I_i$, we have
$$ \mu(f_{|_{I_{i}}}(S))
=
\mu(f(S))
=
\int_S \exp\left(P(f,\varphi) - \varphi\right) \ d\mu
=
\int_S \exp\left(P(f,\varphi) - \varphi\right) \ d\mu_{|_{I_i}}. $$
Hence, if we define $\nu_i \= \left( f_{|_{I_i}}^{-1} \right)_*\mu_{|_{f(I_i)}}$, then  $\frac{d\nu_i}{d \mu_{|_{I_i}}}= \exp(P(f,\varphi)-\varphi)$ on a subset of~$I_i$ of full measure with respect to~$\mu$.
It follows that for every~$\psi$ in~$\Lspace^1(\mu)$, we have
\begin{align*}
\int_{f(I_i)} \psi \circ f_{|_{I_i}}^{-1} \ d\mu_{|_{f(I_i)}}
=
\int_{I_i} \psi \ d\nu_i
=
\int_{I_i} \psi\exp\left(P(f,\varphi)-\varphi\right) \ d \mu.
\end{align*}
Replacing~$\psi$ by $\psi \exp(-P(f,\varphi)+\varphi)$ above, we obtain
\begin{equation}
\label{e:integralonsmall}
\exp\left(-P(f,\varphi)\right) \int_{f(I_i)}
\left(\psi\exp(\varphi)\right) \circ f_{|_{I_i}}^{-1} \ d\mu
=
\int_{I_i} \psi \ d\mu.
\end{equation}
It follows that
\begin{align*}
\int_{J(f)} \hcL_\varphi(\psi) \ d\mu
&=
\exp\left(-P(f,\varphi)\right) \sum_{i=1}^N
\int_{f(I_i)}\left(\psi\exp(\varphi)\right) \circ f_{|_{I_i}}^{-1} \ d\mu
\\ & =
\sum_{i=1}^N\int_{I_i} \psi \ d \mu
\\ & =
\int_{J(f)}\psi \ d\mu.
\end{align*}
The proof of the lemma is complete.
\end{proof}

To prove Theorem~\ref{t:spectral gap}, let~$\mu$ be an atom-free
$\exp(P(f,\varphi)-\varphi)$-conformal measure for~$f$ given by
Theorem~\ref{t:conformal measure}.
By~\eqref{e:1} the function~$\varphi$ is of bounded $(1/\alpha)$-variation.
Since~$\varphi$ is continuous and hence bounded, it follows that~$C\=\sup_{J(f)}\exp(\varphi)<+\infty$.
So for all~$x_1$ and~$x_2$ in~$[0, 1]$, we have
$$ |\exp(\varphi)(x_1)-\exp(\varphi)(x_2)|
\le
C|\varphi(x_1)-\varphi(x_2)|. $$
This implies that~$\exp(\varphi)$, and so~$g$, is of bounded $(1/\alpha)$-variation.
On the other hand, our hypothesis~$\sup_{J(f)}\varphi < P(f,\varphi)$ implies~$\sup_{J(f)} g < 1$.
Using that~$\mu$ is a $(g^{-1})$-conformal measure for~$f$, properties~H$1$ and~H$2$ in~\S\ref{ss:quasicompactness} hold with~$T$, $X$, and~$m$ replaced by~$f$, $J(f)$, and~$\mu$, respectively.
Moreover, Lemma~\ref{l:operator} implies property~H$3$ of~\S\ref{ss:quasicompactness} also holds. 
Therefore, for each~$\talpha$ in~$(0, \alpha]$ the assertions of Theorem~\ref{t:quasicompactness} and Corollary~\ref{c:quasicompactness} hold with~$p = 1 / \talpha$, and with~$m$, $T$, and~$X$ replaced by~$\mu$, $f$, and~$J(f)$, respectively.
Let~$A$ be the constant given by Theorem~\ref{t:quasicompactness} and consider the corresponding space~$\Hspace^{\talpha, 1}(\mu)$.

\begin{proof}[Proof of Theorem~\ref{t:spectral gap}]
By the considerations above, we only need to prove the uniqueness of the equilibrium state and the assertion about the analyticity of the pressure function.
For the former, note that for each equilibrium state~$\nu$ of~$f$ for the potential~$\varphi$, we have
$$ h_{\nu}(f)
=
P(f, \varphi) - \int \varphi \ d \nu
\ge
P(f, \varphi) - \sup_{J(f)} \varphi > 0. $$
Assuming~$\nu$ ergodic, by Ruelle's inequality we have $\chi_{\nu}(f) > 0$.
Then the uniqueness follows from the hypothesis that~$f$ is
topologically exact on~$J(f)$ and from~\cite[Theorem~$6$]{Dob1304}.

It remains to show that the function~$t \mapsto P(f, \varphi + t \chi)$ is real analytic on a neighborhood of~$t = 0$.
For each~$\tau$ in~$\C$, let~$\sL_\tau$ be the operator defined in
part~$3$ of Corollary~\ref{c:quasicompactness} with~$T$ replaced
by~$f$ and~$X$ by~$J(f)$.
On the other hand, for each~$t$ in~$\R$ put
$$ \varphi_t \= \varphi + t \chi
\text{ and }
g_t \= \exp(\varphi_t - P(f, \varphi_t)), $$
and note that the operator~$\exp(- P(f, \varphi_t)) \cL_t$ coincides with the operator~$\sL_g$ defined in~\eqref{e:otor} with~$g$ replaced by~$g_t$, $T$ replaced by~$f$, and~$X$ replaced by $J(f)$.
Let~$\talpha$ in~$(0, 1]$ be sufficiently small so that both~$\varphi$ and~$\chi$ are H{\"o}lder continuous of exponent~$\talpha$, let~$p_1$ and $p_2$ in $(\sup_{J(f)}\varphi, P(f, \varphi))$ be such that $p_1<p_2$, and let~$\eps_0>0$ be small enough so that
$$ \eps_0\sup_{J(f)}|\chi|
<
\min \left\{ p_2 - p_1, P(f,\varphi) - p_2 \right\}. $$
Note that by our choice of~$\eps_0$, for every~$t$ in~$(-\eps_0, \eps_0)$ we have on~$J(f)$ that
$$ \varphi_t
>
\varphi - (P(f,\varphi)-p_2) \1_{J(f)}, $$
so
\begin{multline}
\label{e:being supra pressure}
P(f, \varphi_t) \ge
P(f, \varphi - (P(f,\varphi) - p_2) \1_{J(f)})
\\=
P(f, \varphi) - (P(f,\varphi) - p_2)
=
p_2
>
\sup_{J(f)}\varphi + \sup_{J(f)}t\chi
\ge
\sup_{J(f)}\varphi_t.
\end{multline}
Thus $\sup_{J(f)} g_t < 1$ and by Theorem~\ref{t:conformal measure} there is an atom-free $\exp(P(f,\varphi_t)-\varphi_t)$-conformal measure~$\mu_t$ for~$f$.
Since the function~$g_t$ is of bounded~$(1/\talpha)$-variation, it follows that properties~H$1$ and~H$2$ in~\S\ref{ss:quasicompactness} hold with~$p = 1/\talpha$, and with~$g$, $T$, $X$, and~$m$ replaced by~$g_t$, $f$, $J(f)$, and~$\mu_t$, respectively.
Moreover, Lemma~\ref{l:operator} implies that property~H$3$ of~\S\ref{ss:quasicompactness} also holds. 
We can thus apply part~$1$ of Corollary~\ref{c:quasicompactness} to
conclude that~$\exp(P(f, \varphi_t))$ is equal to the spectral radius
of~$\sL_t$.
Moreover, by part~$3$ of the same corollary, the function~$t \mapsto
\exp(P(f, \varphi_t))$ is real analytic on~$(-\varepsilon_0, \varepsilon_0)$.
The proof of Theorem~\ref{t:spectral gap} is thus complete.
\end{proof}

\appendix
\section{Hyperbolic potentials and the bounded range condition}
\label{s:big oscillation}
Let~$X$ be a compact metric space and~$T : X \to X$ be a continuous
map.
Denote by~$\htop(T)$ the \emph{topological entropy of~$T$}.
Recall that~$2$ continuous functions~$\varphi : X \to \R$ and~$\tvarphi : X \to \R$ are \emph{co-homologous}, if there is a continuous function~$\chi : X \to \R$ such that
$$ \tvarphi = \varphi + \chi - \chi \circ T. $$

It is easy to see that every continuous potential~$\varphi : X \to \R$ satisfying the \emph{bounded range} condition:
\begin{equation}
\label{e:bounded range}
\sup_X \varphi - \inf_X \varphi < \htop(T),
\end{equation}
also satisfies
\begin{equation}
  \label{e:first step hyperbolic}
\sup_X \varphi < P(T, \varphi),
\end{equation}
and it is therefore hyperbolic for~$T$.

The purpose of this section is to show that, under a fairly general condition on~$T$, there is a potential~$\varphi$ satisfying~\eqref{e:first step hyperbolic} that is not co-homologous to any potential~$\tvarphi$ satisfying~\eqref{e:bounded range} with~$\varphi$ replaced by~$\tvarphi$.
When~$f$ is a map in~$\sA$ that is topologically exact on~$J(f)$, this general condition is easily seen to be satisfied when~$T = f|_{J(f)}$.

\begin{lemm}
\label{l:big oscillation}
Let~$X$ be a compact metric space, let~$T : X \to X$ be a continuous map, and let~$h > 0$ be given.
Suppose there are disjoint compact subsets~$X'$ and~$X''$ of~$X$ that are forward invariant by~$T$, and such that~$\htop(T|_{X'}) > 0$.
Let $\varphi : X \to (- \infty, 0]$ be a continuous function that is constant equal to~$0$ on~$X'$, and such that $\sup_{X''} \varphi < - h$.
Then
$$ \sup_X \varphi < P(T, \varphi), $$
and for every continuous function~$\tvarphi : X \to \R$ that is co\nobreakdash-homologous to~$\varphi$, we have
\begin{equation}
\label{e:big oscillation}
\sup_X \tvarphi - \inf_X \tvarphi > h.
\end{equation}
\end{lemm}
\begin{proof}
By the variational principle there is a probability measure~$\nu'$ on~$X$ that is supported on~$X'$, that is invariant by~$T$, and such that~$h_{\nu'}(T) = h_{\nu'} \left( T|_{X'} \right) > 0$.
Then we have
$$ P(T, \varphi)
\ge
h_{\nu'}(T) + \int \varphi \ d \nu'
=
h_{\nu'}(T)
>
0
= \sup_X \varphi. $$
On the other hand, if~$\tvarphi : X \to \R$ is a continuous function that is co-homologous to~$\varphi$, then we have
$$ \sup_X \tvarphi
\ge
\int \tvarphi \ d \nu'
=
\int \varphi \ d \nu'
=
0. $$
Moreover, if~$\nu''$ is a probability measure on~$X''$ that is invariant by~$T$, then
\begin{equation*}
\inf_X \tvarphi
\le
\int \tvarphi \ d \nu''
=
\int \varphi \ d \nu''
=
\int_{X''} \varphi \ d \nu''
\\ \le
\sup_{X''} \varphi
< - h.
\end{equation*}
Together with the inequality~$\sup_X \tvarphi \ge 0$ shown above, this implies~\eqref{e:big oscillation}.
\end{proof}

\bibliographystyle{alpha}

\end{document}